\documentclass[12pt]{amsart}
\usepackage{amsmath}
\usepackage{mathrsfs}
\usepackage{graphicx}
\def\ol{\overline}
\def\e{\epsilon}

\def\a{{\alpha}}

\def\ba{{\bar\a}}
\def\bk{{\bar k}}
\def\bl{{\bar l}}

\def\p{\partial}
\def\cut{{\rm cut}}
\newcommand\R{{\mathbb R}}

\newcommand\C{{\mathbb C}}

\def\ol{\overline}

\def\a{{\alpha}}

\def\ii{{\sqrt{-1}}}

\def\dbar{\overline{\partial }}
\def\ddbar{\partial\overline\partial}
\def\vv<#1,#2>{\left\langle#1,#2\right\rangle}
\def\iddbar{{\sqrt{-1}\partial \overline\partial}}
\newtheorem{thm}{Theorem}[section]

\newtheorem{lem}{Lemma}[section]

\newtheorem{cor}{Corollary}[section]
\theoremstyle{definition}
\newtheorem{defn}{Definition}[section]
\theoremstyle{remark}
\newtheorem{rem}{Remark}
\numberwithin{equation}{section}
\def\bi{{\bar i}}

\def\bj{{\bar j}}
\def\ord{{\rm ord}}
\def\Hess{{\rm Hess}}
\def\tn{{\mathbf{tn}}}
\def\tr{{\rm tr}}

\def\D{\displaystyle}
\begin{document}
\title{Three circle theorem on almost Hermitian manifolds and applications}

\author{Chengjie Yu$^1$}
\address{Department of Mathematics, Shantou University, Shantou, Guangdong, 515063, China}
\email{cjyu@stu.edu.cn}
\author{Chuangyuan Zhang}
\address{Department of Mathematics, Shantou University, Shantou, Guangdong, 515063, China}
\email{12cyzhang@stu.edu.cn}
\thanks{$^1$Research partially supported by GDNSF with contract no. 2021A1515010264 and NNSF of China with contract no. 11571215.}
\renewcommand{\subjclassname}{%
  \textup{2010} Mathematics Subject Classification}
\date{}
\begin{abstract}
In this paper, we extend Gang Liu's three circle theorem for K\"ahler manifolds to almost Hermitian manifolds. As applications of the three circle theorem, we obtain sharp dimension estimates for spaces of holomorphic functions of polynomial growth and rigidity for the estimates, Liouville theorems for pluri-subharmonic functions of sub-logarithmic growth and Liouville theorems for holomorphic functions of Cheng-type, and fundamental theorem of algebra on almost Hermitian manifolds with slightly negative holomorphic sectional curvature. The results are generalization of Liu's results and some of them are new even for the K\"ahler case. We also discuss the converse of the three circle theorem on Hermitian manifolds which turns out to be rather different with the K\"ahler case. In order to obtain the three circle theorem on almost Hermitian manifolds, we also establish a general maximum principle in the spirit of Calabi's trick so that a general three circle theorem is a straight forward consequence of the general maximum principle. The general maximum principle and three circle theorem established may be of independent interests. 
\end{abstract}
\maketitle\markboth{Yu \& Zhang}{Three circle theorem on almost Hermitian manifolds}
\maketitle
\section{Introduction}
One of the recent breakthroughs in the study of the structures of nonnegatively curved complete noncompact K\"ahler manifolds is a series of works \cite{Liu,Liu-JDG,Liu-Ann,Liu-Uni} by  Liu which finally sovled Yau's finite generation conjecture and Yau's uniformization conjecture under the assumption of nonnegative holomorphic bisectional curvature and maximal volume growth. The later conjecture was also solved by Tam and his collaborators in a series of works \cite{CT,HT,LT} using K\"ahler-Ricci flow following the pioneer work of Shi \cite{Shi}. Mok \cite{Mok} and Chen-Tang-Zhu \cite{CTZ} also solved the later conjecture on complete noncompact K\"ahler surfaces under certain geometric asumptions. It is an interesting problem to extend the aforementioned results to complete noncompact Hermitian manifolds. There is a similar problem in the compact case. It was shown by Siu-Yau \cite{SY} and Mori \cite{Mori} independently that compact K\"ahler manifolds with positive holomorphic bisectional curvature must be biholomorphic to the complex projective space. In fact, Mori's result is more general. He solved the Hartshorne conjecture in algebraic geometry which at least implies that a compact Hermitian manifold with positive curvature in the sense of Griffiths must be also biholomorphic to the complex projective space. It is a folklore problem in geometric analysis to give an analytic proof of the aforementioned fact.

For analytic methods dealing with the Hermitian case, similar to K\"ahler-Ricci flow, there are Hermitian curvature flow introduced by Streets-Tian \cite{ST}, and Chern-Ricci flow introduced by Gill \cite{Gill} and Tosatti-Weinkove \cite{TW}. It seems that the Hermitian curvature flow is more like the K\"ahler-Ricci flow because it preserves the signs of certain interested curvatures (See \cite{Us1,Us2}) while the Chern-Ricci flow does not (See \cite{Yang}).  Analogous to the K\"ahler case, there must be a similar approach to Liu's for the Hermitian case. The starting point of Liu's approach is an elegant and clever generalization of Hadamard's three circle theorem to complete noncompact K\"ahler manifolds with nonnegative holomorphic sectional curvature in \cite{Liu}. So, the first step to extend Liu's approach in the Hermitian case is to extend the three circle theorem to Hermitian manifolds. This is the motivation of our work. In fact, we are able to extend the three circle theorem to the more general almost Hermitian case.

Let's recall Liu's (\cite{Liu}) three circle theorem first.
\begin{thm}[Liu \cite{Liu}]\label{thm-Liu}
Let $M$ be a complete K\"ahler manifold. Then $M$ satisfies the three circle theorem if and only if the holomorphic sectional curvature is nonnegative. Here, a complete K\"ahler manifold is said to satisfy the three circle theorem if for any $p\in M$, $R>0$  and nonzero holomorphic function $f$ on $B_p(R)$, $\log M(f,r)$ is a convex function with respect to $\log r$ for $r\in (0,R)$ where $M(f,r)=\max_{x\in B_p(r)}|f(x)|$.
\end{thm}
In \cite{Liu}, Liu also extend the "if" part of the three circle theorem above to much more general K\"ahler manifolds (See \cite[Theorem 8]{Liu} ).  The purpose of this paper is to extend Liu's results in \cite{Liu} on K\"ahler manifolds to Hermitian manifolds or almost Hermitian manifolds and give some applications. We would also like to mention that an analogue of Liu's three circle theorem in the case of CR manifolds was obtained in \cite{CHL}.

Recall that an almost Hermitian manifold $(M^{2n},J,g)$ is a manifold $M$ of real dimension $2n$ equipped with an almost complex structure $J$ and a compatible Riemannian metric $g$:
$$g(JX,JY)=g(X,Y)$$
for any tangent vectors $X$ and $Y$. When $J$ is integrable, $(M^{2n},J,g)$ is called a Hermitian manifold. The Levi-Civita connection on $(M,J,g)$ denoted as $\nabla$ is torsion free and compatible with $g$ but may not be compatible with $J$. In fact, $\nabla J\not\equiv 0$ unless $(M,J,g)$ is K\"ahler. The Chern connection on $(M,J,g)$ denoted as $D$ is the unique connection on $M$ compatible with $J$ and $g$, and with vanishing $(1,1)$ part of the torsion. More precisely, let
\begin{equation}
\tau(X,Y)=D_XY-D_YX-[X,Y]
\end{equation}
be the torsion of $D$. Then,
\begin{equation}\label{eq-vanishing-1-1}
\tau(\xi,\bar\eta)=0
\end{equation}
for any $(1,0)$-vectors $\xi$ and $\eta$. Or equivalently,
\begin{equation}\label{eq-vanishing-1-1-J}
\tau(X,JY)=\tau(JX,Y)
\end{equation}
for any tangent vectors $X$ and $Y$. All the linear operations on real vector fields are extended to complex vector fields by complex linear extension. Moreover, for any $\omega\in A^{p,q}(M)$, define $\p \omega:=(d\omega)^{p+1,q}$ and $\dbar\omega:=(d\omega)^{p,q+1}$.  The same as on complex manifolds, a complex-valued function $f$ on $M$ is said to be holomorphic if $\dbar f=0$, and a real-valued function $u$ on $M$ is said to be pluri-subharmonic if $\iddbar u$ is a nonnegative $(1,1)$-form.

By examining Liu's work \cite{Liu}, we know that a key observation in Liu's argument is a comparison result for $\Hess(r)(J\nabla r,J\nabla r)$ where $\Hess$ means the Hessian operator with respect to Levi-Civita connection $\nabla$ and $r$ is the distance function with respect to some fixed point. On almost Hermitian manifolds, such a comparison was essentially already known in Gray's work \cite{Gray} using Levi-Civita connection and in the first named author's work \cite{Yu2} using Chern connection. Such a comparison was also obtained for Hermitian manifolds in the work \cite{CY} of Chen-Yang. More precisely, all of these works contain the following conclusion in an implicit form:
\begin{equation}\label{eq-comparison}
L[\log r]\leq 0
\end{equation}
within the cut-locus of the fixed point under the curvature assumption
\begin{equation}\label{eq-holo-sec-L}
R^L(X,JX,JX,X)-\|(\nabla_X J)X\|^2\geq 0
\end{equation}
for any tangent vector $X$ where $R^L$ is the curvature tensor for the Levi-Civita connection (see \cite{Gray}), or equivalently
\begin{equation}\label{eq-holo-sec-C}
R_{1\bar 11\bar 1}-\sum_{i=2}^n|\tau_{i1}^1+\tau_{i1}^{\bar 1}|^2\geq 0
\end{equation}
for any unitary $(1,0)$-frame $e_1,e_2,\cdots,e_n$ where $R$ is the curvature tensor for the Chern connection (see \cite{CY} for the Hermitian case and \cite{Yu2} for the almost Hermitian case). Here,
\begin{equation}
L[u]:=\Hess(u)(\nabla u,\nabla u)+\Hess(u)(J\nabla u,J\nabla u).
\end{equation}
The equivalence of \eqref{eq-holo-sec-L} and \eqref{eq-holo-sec-C} can be seen by using the curvature identities derived by the first named author in \cite{Yu1} (See Lemma \ref{lem-curv-holo-sec} for details).  According to these results, we introduce the following notion.
\begin{defn}\label{def-holo-sec}
Let $(M^{2n},J,g)$ be an almost Hermitian manifold and $\xi$ be a unit $(1,0)$-vector. Define the holomorphic sectional curvature $H(\xi)$ along the direction $\xi$ as
\begin{equation}
H(\xi):=R_{1\bar 11\bar 1}-\sum_{i=2}^n|\tau_{i1}^1+\tau_{i 1}^{\bar 1}|^2
\end{equation}
where we have fixed a unitary frame $(e_1,e_2,\cdots,e_n)$ with $e_1=\xi$. Let $\xi=\frac{1}{\sqrt 2}(X-\ii JX)$ with $X$ a unit real tangent vector. Then, we also write $H(\xi)$ as $H(X)$.
\end{defn}
By following the argument in \cite{Yu2}, we are able to obtain a general comparison result.
\begin{thm}\label{thm-comparison}
Let $(M^{2n},J,g)$ be a complete almost Hermitian manifold  and $o\in M$. Let $h\in C^1((0,R))$ be such that
\begin{enumerate}
\item $\D\lim_{t\to 0^+}th(t)=1$, and
\item for any normal minimal geodesic $\gamma:[0,l]\to B_o(R)$ starting at $o$,
\begin{equation}\label{eq-Riccati-h}
h'(t)+h^2(t)+H(\gamma'(t))\geq 0
\end{equation}
for $t\in (0,l]$.
\end{enumerate}
 Let $v\in C^2((0,R))$ be such that
\begin{enumerate}
\item $v'(t)>0$ for $t\in (0,R)$;
\item $v''(t)+hv'(t)\leq 0$ for $t\in (0,R)$.
\end{enumerate}
Then
$$L[v(r_o)]\leq 0$$
in $B'_o(R):=B_o(R)\setminus\{o\}$, in the sense of barrier, where $r_o(x)=r(o,x)$. In particular, if the holomorphic sectional curvature of $(M,J,g)$ in $B_o(R)$ is not less than a constant $K$, then
\begin{equation}\label{eq-comparison-K}
L[\log(\tn_K(r_o/2))]\leq 0
\end{equation}
in $B'_o(R)$, in the sense of barrier, where
\begin{equation}
\tn_K(t)=\left\{\begin{array}{cl}\frac{1}{\sqrt K}\tan(\sqrt Kt)&K>0\\
t&K=0\\
\frac{1}{\sqrt{-K}}\tanh(\sqrt{-K}t)&K<0.
\end{array}\right.
\end{equation}
\end{thm}
Here, because the distance function may not be smooth and motivated by Calabi's trick \cite{Ca}, we say that a continuous function $v$ on a open subset $\Omega$ of $M$ satisfies
$L[v]\leq 0$
in the sense of barrier. If for any $\e>0$ and $x\in \Omega$, there is a smooth function $v_{x,\e}$ defined on an open neighborhood $U_{x,\e}\subset \Omega$ of $x$ such that $v\leq v_{x,\e}$ on $U_{x,\e}$, $v(x)=v_{x,\e}(x)$ and $L[v_{x,\e}](x)\leq \e$.

The three circle theorem we obtained for almost Hermitian manifolds is as follows.

\begin{thm}\label{thm-three-circle-AH}
Let $(M^{2n},J,g)$ be a complete noncompact almost Hermitian manifold  and $o\in M$. Let $h\in C^1((0,R))$ be such that
\begin{enumerate}
\item $\D\lim_{t\to 0^+}th(t)=1$, and
\item for any normal minimal geodesic $\gamma:[0,l]\to B_o(R)$ starting at $o$,
\begin{equation}\label{eq-Riccati-h}
h'(t)+h^2(t)+H(\gamma'(t))\geq 0
\end{equation}
for $t\in (0,l]$.
\end{enumerate} Let $v\in C^2((0,R))$ be such that
\begin{enumerate}
\item $v'(t)>0$ for $t\in (0,R)$;
\item $v''(t)+hv'(t)\leq 0$ for $t\in (0,R)$.
\end{enumerate}
Then, for any $u\in C^2(B_o(R))$ with $L[u]\geq 0$, $M_o(u,r)$ is a convex function of $v(r)$ for $r\in (0,R)$. Moreover, if the holomorphic sectional curvature of $(M,J,g)$ in $B_o(R)$ is not less than a constant $K$, then $M_o(u,r)$ is a convex function of $\log\tn_K(r/2)$ for $r\in (0,R)$. In particular,
\begin{enumerate}
\item if $u\in C^2(B_o(R))$ is a pluri-subharmonic function, then $M_o(u,r)$ is a convex function of $v(r)$ for $r\in (0,R)$;
\item for any holomorphic functions $f_1,f_2,\cdots,f_k$ on $B_o(R)$, $\log M_o(|f_1|^2+|f_2|^2+\cdots+|f_k|^2,r)$ is a convex function of $v(r)$ for $r\in (0,R)$;
\item if the holomorphic sectional curvature of $(M,J,g)$ in $B_o(R)$ is not less than a constant $K$ and $u\in C^2(B_o(R))$ is a pluri-subharmonic function, then $M_o(u,r)$ is a convex function of $\log\tn_k(r/2)$ for $r\in (0,R)$;
\item if the holomorphic sectional curvature of $(M,J,g)$ in $B_o(R)$ is not less than a constant $K$, then $\log M_o(|f_1|^2+|f_2|^2+\cdots+|f_k|^2,r)$ is a convex function of $\log\tn_k(r/2)$ for $r\in (0,R)$ and any holomorphic functions $f_1,f_2,\cdots,f_k$ on $B_o(R)$.
\end{enumerate}
 \end{thm}
  The basic idea of proving the last theorem is similar to that in \cite{Liu}. However, motivated by the work \cite{Ca} of Calabi, we deal with the topic in a more systematic way by extracting the arguments in \cite{Liu} to formulate a general maximum principle (See Theorem \ref{thm-max-principle}) which seems have not been mentioned before, so that the general three circle theorem (See Theorem \ref{thm-general-three-circle}) becomes a straight forward consequence of the maximum principle. The general maximum principle and three circle theorem we derived may be useful in other applications. For example, before proving Theorem \ref{thm-three-circle-AH}, we obtain the following maximum principle for functions $u$ satisfying $L[u]\geq 0$ which is different with the maximum principle for subharmonic functions and seems not been mentioned before even for the K\"ahler case.
\begin{thm}\label{thm-max-principle-AH}
Let $(M^{2n}, J,g)$ be a complete noncompact almost Hermitian manifold, $\Omega$ be a precompact open subset in $M$ and $u\in C^2(\Omega)\cap C^0(\ol\Omega)$ such that $L[u]\geq 0$. Then, $u\leq \max_{\p\Omega}u$.
\end{thm}

For convenience of later application of Theorem \ref{thm-three-circle-AH}, we introduction the following two notions.
\begin{defn}
Let $q(t)$ be a nonnegative continuous function on $[0,+\infty)$ and $h$ be the solution of $$h'+h^2-q=0$$ with $\D\lim_{t\to 0^+}th(t)=1$. Then, the solution $v$ of $$v''+hv'=0$$
with $\D\lim_{t\to 0^+}\frac{v(t)}{\log t}=1$ is called the function associated to $q$. Define
\begin{equation}
I(q)=\exp\left(\int_0^{+\infty} tq(t)dt\right)
\end{equation}
when $\D\int_0^{+\infty}t q(t)dt<+\infty$.
\end{defn}
\begin{defn}
Let $(M^{2n},J,g)$ be a complete noncompact almost Hermitian manifold and $o\in M$. For $r>0$, define
\begin{equation*}
H_o(r)=\inf\{H(\gamma'(r))\ |\  \gamma:[0,r]\to M {\rm \ is\ a\ normal\ minimal\ geodesic\ with\  } \gamma(0)=o\}.
\end{equation*}
and
\begin{equation*}
H_o(0)=\min\{H(X)\ |\ X\in T_oM\ {\rm with}\ \|X\|=1\}.
\end{equation*}
\end{defn}

The first application of Theorem \ref{thm-three-circle-AH} is the following Liouville theorem for functions satisfying $L[u]\geq 0$ on almost Hermitian manifolds with slightly negative holomorphic sectional curvature. Such a Liouville theorem for pluri-subharmonic function on complete K\"ahler manifolds with nonnegative holomorphic bisectional curvature was first obtained by Ni-Tam \cite{NT}. Liu \cite{Liu} relaxed the curvature assumption to nonnegative holomorphic sectional curvature and extend it to functions satisfying $L[u]\geq 0$. Recently, under the different curvature assumption of nonnegative orthogonal holomorphic bisectional curvature and nonnegative Ricci curvature, Ni-Niu \cite{NN} also obtained the Liouville theorem for pluri-subharmonic functions on complete noncompact K\"ahler manifolds.
\begin{cor}\label{cor-Liouville-psh-2}
Let $(M^{2n},J,g)$ be a complete noncompact almost Hermitian manifold. Suppose for each $p\in M$, there is a nonnegative continuous function $q_p$ on $[0,+\infty)$ such that
\begin{enumerate}
\item $H_p(r)\geq -q_p(r)$ for any $r\geq0$, and
\item $\D\int_0^{+\infty}rq_p(r)dr<+\infty$.
\end{enumerate}
Let $u\in C^2(M,\R)$ satisfy $L[u]\geq 0$ and
\begin{equation}\label{eq-sub-log}
\liminf_{r\to+\infty}\frac{M_o(u,r)}{\log r}\leq 0.
\end{equation}
where $o$ is a fixed point on $M$. Then, $u$ must be a constant function. In particular, any pluri-subharmonic function $u$ satisfying \eqref{eq-sub-log} must be a constant function.
\end{cor}

Similarly as in the K\"ahler case, we denote the ring of holomorphic functions on $M$ as  $\mathcal O(M)$. A function $f$ on $M$ is said to be of polynomial growth if there are constants $C,d>0$ such that
\begin{equation}\label{eq-poly}
|f(x)|\leq C(1+r(x)^d),\ \forall x\in M.
\end{equation}
Denote the ring of holomorphic functions on $M$ of polynomial growth as $\mathcal P(M)$. Denote $\mathcal O_d(M)$ the space of holomorphic functions such that \eqref{eq-poly} holds for some $C>0$. For each nonzero $f\in \mathcal P(M)$, define
\begin{equation}
\deg f=\inf\{d>0\ |\ f\in \mathcal O_d(M)\}
\end{equation}
which is called the degree of $f$. Then, similar to \cite{Liu}, we have the following consequence of Theorem \ref{thm-three-circle-AH}.
\begin{thm}\label{thm-mono-2}
Let $(M^{2n},J,g)$ be a complete noncompact almost Hermitian manifold and $o\in M$. Suppose there is a nonegative continuous function $q(t)$ on $[0,+\infty)$ such that  \begin{enumerate}
\item $H_o(r)\geq -q(r)$ for any $r\geq0$, and
\item $\D\int_0^{+\infty}rq(r)dr<+\infty$.
\end{enumerate}
Then,
\begin{enumerate}
\item $\displaystyle f\in \mathcal P(M)$ if and only if $\displaystyle\liminf_{r\to+\infty}\frac{\log M_o(|f|,r)}{\log r}<+\infty$;
\item for any $f\in \mathcal P(M)$, $$\deg f=\liminf_{r\to+\infty}\frac{\log M_o(|f|,r)}{\log r}=\limsup_{r\to+\infty}\frac{\log M_o(|f|,r)}{\log r};$$
\item for any $f\in \mathcal P(M),$ $$\log M_o(|f|,r)-I(q)\deg f\cdot v(r)$$ is decreasing on $r$. In particular, when $M$ has nonnegative holomorphic sectional curvature, $\log M_o(|f|,r)-\deg f\cdot \log r$ is decreasing;
\item for any nonzero holomorphic function $f$ on $B_o(R)$, $$\log M_o(|f|,r)-\ord_o(f)\cdot v(r)$$ is increasing for $r\in (0,R)$. In particular, when $M$ has nonnegative holomorphic sectional curvature $\log M_o(|f|,r)-\ord_o(f)\cdot \log r$ is increasing;
\item for any nonzero $f\in \mathcal P(M)$, $$\ord_o(f)\leq I(q)\deg f.$$ In particular, when $M$ has nonnegative holomorphic sectional curvature, $$\ord_o(f)\leq \deg f.$$
\end{enumerate}
Here $\ord_o(f)$ means the vanishing order of $f$ at $o$, and $v$ is the function associated to $q$.
\end{thm}
We would like to mention that the sharp vanishing order estimate
$$\ord_o(f)\leq \deg f$$
on complete noncompact K\"ahler manifolds was first obtained by Ni \cite{Ni} under the assumptions of nonnegative holomorphic bisectional curvature and maximal volume growth. Later, the assumption of maximal volume growth was removed by Chen-Fu-Yin-Zhu \cite{CFYZ}. Liu \cite{Liu} relaxed the curvature assumption to nonnegative holomorphic sectional curvature. For vanishing order estimate which is not sharp, it was first obtained by Mok \cite{Mok} under some more restrictive geometric assumptions.

As a consequence of Theorem \ref{thm-mono-2}, we have the following Liouville theorem of Cheng-type for holomorphic functions.
\begin{cor}\label{cor-Liouville-holo-2}
Let $(M^{2n},J,g)$ be a complete noncompact almost Hermitian manifold and $o\in M$. Suppose there is a nonnegative continuous function $q(t)$ on $[0,+\infty)$ such that  \begin{enumerate}
\item $H_o(r)\geq -q(r)$ for any $r\geq0$, and
\item $\D\int_0^{+\infty}rq(r)dr<+\infty$.
\end{enumerate} Let $f$ be a holomorphic function on $M$ such that
\begin{equation}\label{eq-sub-linear}
\lim_{r\to +\infty}r^{-\frac{1}{I(q)}}{M_o(|f|,r)}=0.
\end{equation}
Then, $f$ must be a constant function. In particular, when $M$ has nonnegative holomorphic sectional curvature and $f$ is a holomorphic function on $M$ of sub-linear growth, $f$ must be constant.
\end{cor}

Note that Cheng's Liouville theorem is for harmonic functions and under the curvature assumption of nonnegative Ricci curvature. Here, the curvature assumption in Corollary \ref{cor-Liouville-holo-2} is rather different. By Corollary \ref{cor-Liouville-holo-2}, we have the following result analogous to fundamental theorem of algebra on almost Hermitian manifolds. It seems that this conclusion was not mentioned before even for the K\"ahler case.
\begin{cor}\label{cor-fund-holo-2}
Let $(M^{2n},J,g)$ be a complete noncompact simply connected almost Hermitian manifold and $o\in M$. Suppose there is a nonnegative continuous function $q(t)$ on $[0,+\infty)$ such that  \begin{enumerate}
\item $H_o(r)\geq -q(r)$ for any $r\geq0$, and
\item $\D\int_0^{+\infty}rq(r)dr<+\infty$.
\end{enumerate}  Then, for any nonconstant holomorphic function $f\in \mathcal P(M)$, $Z(f)\neq\emptyset$ where $Z(f)=\{x\in M\ |\ f(x)=0\}$.
\end{cor}

When the almost Hermitian manifold has nonnegative holomorphic sectional curvature, by taking universal cover, we can drop the assumption of simply connectivity on $M$.
\begin{cor}\label{cor-fund-holo}
Let $(M^{2n},J,g)$ be a complete noncompact almost Hermitian manifold with nonnegative holomorphic sectional curvature. Then, for any nonconstant holomorphic function $f\in \mathcal P(M)$, $Z(f)\neq\emptyset$.
\end{cor}
Moreover,  as a consequence of the vanishing order estimate in  Theorem \ref{thm-mono-2}, we have the following estimate for the dimension of $\mathcal O_d(M)$ which is sharp when the holomorphic sectional curvature is nonnegative. This sharp dimension estimate and its rigidity on complete noncompact K\"ahler manifolds was first obtained by Ni \cite{Ni} under the assumptions of nonnegative holomorphic bisectional curvature and maximal volume growth. Later, Chen-Fu-Yin-Zhu \cite{CFYZ} removed the assumption of maximal volume growth. The curvature assumption was relaxed to nonnegative holomorphic sectional curvature by Liu \cite{Liu}.
\begin{thm}\label{thm-dim}
Let $(M^{2n},J,g)$ be a complete noncompact almost Hermitian manifold and $o\in M$. Suppose there is a nonnegative continuous function $q(t)$ on $[0,+\infty)$ such that  \begin{enumerate}
\item $H_o(r)\geq -q(r)$ for any $r\geq0$, and
\item $\D\int_0^{+\infty}rq(r)dr<+\infty$.
\end{enumerate}
 Then, for any positive number $d$,
\begin{equation}\label{eq-dim-g}
\dim \mathcal O_d(M)\leq \mathcal O_{I(q)d}(\mathbb C^n).
\end{equation}
In particular, when $M$ has nonnegative holomorphic sectional curvature,
\begin{equation}\label{eq-dim-sharp}
\dim \mathcal O_d(M)\leq \mathcal O_{d}(\mathbb C^n).
\end{equation}
Moreover, the equality of \eqref{eq-dim-sharp} holds for some positive integer $d$ if and only if $M$ is biholomorphically isometric to $\C^n$.
\end{thm}
We would like to emphasize that the results we obtained for complete noncompact almost Hermitian manifolds with slightly negative holomorphic sectional curvature can be viewed as an extension of Theorem 11 in \cite{Liu}. Some of them new even for the K\"ahler case. For example, most of the results in this paper also work for complete noncompact K\"ahler manifold satisfying the curvature assumption: $$H_o(r)\geq-\frac{C}{(1+r)^2\log^{1+\e}(2+r)}$$
for some constants $C,\e>0$ which is weaker than the curvature assumption:
$$H_o(r)\geq -\frac{C}{(1+r)^{2+\e}}$$
in \cite[Theorem 11]{Liu}.

The second part of this paper is to consider the converse of the three circle theorem for Hermitian manifolds. We don't consider this for almost Hermitian manifolds because when the almost complex structure is not integrable, the almost Hermitian manifold may support no local holomorphic function. By following the local arguments in \cite{Liu}, we find that the situation for Hermitian manifolds is quite different with the K\"ahler case. We don't get necessary and sufficient conditions.
\begin{defn}
Let $(M^{2n},J,g)$ be an almost Hermitian manifold. If for any $o\in M$, $R>0$, and any nonzero holomorphic function $f$ on $B_o(R)$, $\log M_o(|f|,r)$ is a convex function with respect to $\log r$ for $r\in (0,R)$, we say that $(M,J,g)$ satisfies the three circle theorem.
\end{defn}
By following the idea in \cite{Liu}, we have the following converse of the three circle theorem for Hermitian manifolds.
\begin{thm}\label{thm-con-1-holo}
Let $(M^{2n},J,g)$ be a Hermitian manifold satisfying the three circle theorem. Then
$$R^L(X,JX,JX,X)+3\|(\nabla_XJ)X\|^2\geq 0$$
for any real tangent vector $X$, or equivalently
$$R_{1\bar 11\bar 1}+\sum_{i=2}^n|\tau_{i1}^1|^2\geq 0$$
for any unitary $(1,0)$-frame $e_1,e_2,\cdots,e_n$.
\end{thm}

The rest of the paper is organized as follows. In Section 2, we introduce some preliminaries on almost Hermitian geometry. In Section 3, we give a general maximum principle and a general three circle theorem which set up the framework for latter applications. In Section 4, we obtain the main results for almost Hermitian manifolds by applying the framework set up in Section 3. Finally, in Section 5, we consider converse of three circle theorem on Hermitian manifolds.

{\bf Acknowledgment.} The authors would like to thank Professor Jeffrey D. Streets for informing them the important works \cite{ST,Us1,Us2} on Hermitian curvature flow and X.-K. Yang's work \cite{Yang} on Chern-Ricci flow.
\section{Preliminaries on almost Hermitian geometry}
Let $(M^{2n},J,g)$ be an almost Hermitian manifold of real dimension $2n$.
The difference of the Levi-Civita connection $\nabla$ and the Chern connection $D$ is given by the following identity:
\begin{equation}\label{eq-diff-connections}
\vv<\nabla_YX,Z>=\vv<D_Y X,Z>+\frac{1}{2}[\vv<\tau(X,Y),Z>+\vv<\tau(Y,Z),X>-\vv<\tau(Z,X),Y>]
\end{equation}
for any tangent vector fields $X,Y,Z$ (For a proof of this identity, see \cite{FTY}). Recall that the Nijenhuis tensor for an almost complex manifold is a vector-valued two-form defined as
\begin{equation}\label{eq-Nijenhuis}
N_J(X,Y)=[JX,JY]-J[JX,Y]-J[X,JY]-[X,Y]
\end{equation}
for any tangent vectors $X$ and $Y$. A direct computation gives us that $N_J=0$ if and only if $\tau(\xi,\eta)^{0,1}=0$ for any $(1,0)$-vectors $\xi$ and $\eta$ (See \cite{Yu1} for example). For any smooth function $f$ on $M$, we denote
\begin{equation}
\Hess_D(f)=Ddf
\end{equation}
the Hessian of $f$ with respect to the connection $D$. Note that $\Hess_D(f)$ may not be a symmetric tensor because $D$ may have non-vanishing torsion. In fact, 
\begin{equation}
\Hess_D(f)(X,Y)-\Hess_D(f)(Y,X)=\tau(X,Y)(f).
\end{equation}
for any tangent vectors $X,Y$. Then, by \eqref{eq-vanishing-1-1} and \eqref{eq-vanishing-1-1-J},
\begin{equation}\label{eq-Hess-D-commute}
\Hess_D(f)(\xi,\bar\eta)=\Hess_D(f)(\bar\eta,\xi)
\end{equation}
for any $(1,0)$-vectors $\xi$ and $\eta$, and
\begin{equation}
\Hess_D(f)(X,JX)=\Hess_D(f)(JX,X)
\end{equation}
for any tangent vector $X$. So, for any $\xi=X-\ii JX$ with  $X$ a real tangent vector,
\begin{equation}\label{eq-hess-f-J}
\Hess_D(f)(\xi,\bar\xi)=\Hess_D(f)(X,X)+\Hess_D(f)(JX,JX).
\end{equation}
 Moreover, by direct computation, we have the following identities between $\ddbar f$ and $\Hess_D(f)$.
\begin{lem}\label{lem-ddbar-f}
Let $(M,J,g)$ be an almost Hermitian manifold. Then,
\begin{equation}
\ddbar f(\xi,\bar\eta)=\Hess_D(f)(\xi,\bar \eta)=-\dbar\p f(\xi,\bar\eta)
\end{equation}
for any smooth function $f$ and $(1,0)$-vectors $\xi$ and $\eta$.
\end{lem}
\begin{proof}
By definitions and \eqref{eq-Hess-D-commute},
\begin{equation}
\begin{split}
\ddbar f(\xi,\bar \eta)=&d\dbar f(\xi,\bar\eta)\\
=&\xi(\dbar f(\bar \eta))-\bar\eta(\dbar f(\xi))-\dbar f([\xi,\bar\eta])\\
=&\xi\bar\eta(f)-[\xi,\bar\eta]^{0,1}(f)\\
=&\xi\bar\eta(f)-D_\xi\bar\eta(f)\\
=&\Hess_D(f)(\xi,\bar\eta)
\end{split}
\end{equation}
The other identity can be shown similarly.
\end{proof}
By Lemma \ref{lem-ddbar-f}, a real-valued function $u$ is pluri-subharmonic if and only if
\begin{equation}
\Hess_D(u)(\xi,\bar\xi)\geq0
\end{equation}
for any $(1,0)$-vector $\xi$. By \eqref{eq-hess-f-J}, this is also equivalent to
\begin{equation}
\Hess_D(u)(X,X)+\Hess_D(u)(JX,JX)\geq 0
\end{equation}
for any real tangent vector $X$. Moreover, the same as in the case of complex manifolds, one can construct pluri-subharmonic functions out of holomorphic functions.
\begin{lem}\label{lem-holo-psh}
Let $(M^{2n}, J)$ be an almost complex manifold with real dimension $2n$ and $\Omega\subset M$ be an open subset.
Then,  $\log(|f_1|^2+|f_2|^2+\cdots+|f_k|^2+\e)$ is pluri-subharmonic on $\Omega$ for any holomorphic functions $f_1,f_2,\cdots, f_k$ on $\Omega$ and $\e>0$.
\end{lem}
\begin{proof}
By using Lemma \ref{lem-ddbar-f}, we know that
\begin{equation}
\begin{split}
&\iddbar \log(|f_1|^2+|f_2|^2+\cdots+|f_k|^2+\e)\\
=&\left(\sum_{i=1}^k|f_i|^2+\e\right)^{-2}\left(\left(\sum_{i=1}^k|f_i|^2+\e\right)\sum_{j=1}^k\ii\p f_j\wedge\dbar \ol{f_j}-\ii\left(\sum_{i=1}^k\ol{f_i}\p f_i\right)\wedge\left(\sum_{j=1}^kf_j\dbar\ol{f_j}\right)\right).
\end{split}
\end{equation}
By the Cauchy-Schwarz inequality, we know that
\begin{equation}
\left(\sum_{i=1}^k|f_i|^2\right)\sum_{j=1}^k\ii\p f_j\wedge\dbar \ol{f_j}-\ii\left(\sum_{i=1}^k\ol{f_i}\p f_i\right)\wedge\left(\sum_{j=1}^kf_j\dbar\ol{f_j}\right)\geq 0.
\end{equation}
Then, the conclusion of the lemma follows directly.
\end{proof}

The Hessian operator with respect to the Levi-Civita connection is denoted as
\begin{equation}
\Hess(f)=\nabla df.
\end{equation}
By direct computation using \eqref{eq-diff-connections},
\begin{equation}\label{eq-diff-hess}
\Hess_D(f)(X,Y)=\Hess(f)(X,Y)+\frac{1}{2}[\vv<\tau(X,Y),\nabla
f>+\vv<\tau(Y,\nabla f),X>-\vv<\tau(\nabla f,X),Y>]
\end{equation}
for any tangent vector fields $X$ and $Y$. By \eqref{eq-hess-f-J}, \eqref{eq-diff-hess} and \eqref{eq-vanishing-1-1-J},
\begin{equation}\label{eq-L-f-L}
L[f]=\Hess(f)(\xi_f,\bar\xi_f)=\Hess_D(f)(\xi_f,\bar\xi_f)
\end{equation}
where $\xi_f=\nabla f-\ii J\nabla f$. For each nonzero real tangent vector $X$, we denote the normalization of the $(1,0)$-part of $X$ as $\mathbf{U}(X)$. More precisely,
\begin{equation}
\mathbf{U}(X):=\frac{X-\ii JX}{\|X-\ii JX\|}.
\end{equation}

Next, we come to show the equivalence of the curvature notions formulated in terms of Levi-Civita connection in Gray's work \cite{Gray} and in Definition \ref{def-holo-sec}.
\begin{lem}\label{lem-curv-holo-sec}
Let $(M^{2n},J,g)$ be an almost Hermitian manifold and $\xi=\frac{1}{\sqrt2}(X-\ii JX)$ with $X$ a unit real tangent vector. Then,
\begin{equation}
H(\xi)=R^L(X,JX,JX,X)-\|(\nabla_XJ)(X)\|^2.
\end{equation}
\end{lem}
\begin{proof}
Note that for any tangent vector fields $X$ and $Y$, by \eqref{eq-diff-connections} and \eqref{eq-vanishing-1-1-J},
\begin{equation}\label{eq-nabla-J}
\begin{split}
&\vv<(\nabla_XJ)(X),Y>\\
=&\vv<(\nabla_XJ)(X),Y>-\vv<(D_XJ)(X),Y>\\
=&\vv<\nabla_X(JX)-D_X(JX),Y>-\vv<J(\nabla_XX-D_XX),Y>\\
=&\frac{1}{2}(\vv<\tau(X,Y),JX>-\vv<\tau(Y,JX),X>)+\vv<\nabla_XX-D_XX,JY>\\
=&\frac{1}{2}(\vv<\tau(X,Y),JX>+\vv<\tau(JX,Y),X>)+\vv<\tau(X,JY),X>\\
=&\frac{1}{2}\vv<\tau(X,Y),JX>+\frac32\vv<\tau(JX,Y),X>.
\end{split}
\end{equation}
Let $e_1=\xi,e_2,\cdots,e_n$ be a local unitary frame. Then, $X=\frac{e_1+\ol{e_1}}{\sqrt 2}$ and $JX=\frac{\ii(e_1-\ol{e_1})}{\sqrt 2}$. By \eqref{eq-nabla-J},
\begin{equation}
(\nabla_XJ)X=-\ii\left(\frac12\tau_{\bar 1\bi}^{\bar 1}+\tau_{\bar 1\bi}^{ 1}\right)e_i+\ii\left(\frac12\tau_{ 1i}^{ 1}+\tau_{1i}^{\bar 1}\right)\ol{e_i}.
\end{equation}
Hence
\begin{equation}
\|(\nabla_XJ)X\|^2=2\sum_{i=2}^n\left|\frac12\tau_{ 1i}^{ 1}+\tau_{1i}^{\bar 1}\right|^2.
\end{equation}
Combining this with  Corollary 3.1 in \cite{Yu1}, we have
\begin{equation}
\begin{split}
&R^L(X,JX,JX,X)-\|(\nabla_XJ)X\|^2\\
=&R^L_{1\bar 11\bar 1}-2\sum_{i=2}^n\left|\frac12\tau_{ 1i}^{ 1}+\tau_{1i}^{\bar 1}\right|^2\\
=&R_{1\bar 1 1\bar 1}+\sum_{i=2}^n\left(|\tau_{1i}^{\bar 1}|^2-\frac{1}{2}|\tau_{1i}^1|^2\right)-2\sum_{i=2}^n\left|\frac12\tau_{ 1i}^{ 1}+\tau_{1i}^{\bar 1}\right|^2\\
=&R_{1\bar 11\bar 1}-\sum_{i=2}^n\left|\tau_{ 1i}^{ 1}+\tau_{1i}^{\bar 1}\right|^2\\
=&H(\xi).
\end{split}
\end{equation}
\end{proof}
At the end of this section, we prove Theorem \ref{thm-comparison}, a general comparison theorem for $L[\log r]$.
\begin{proof}[Proof of Theorem \ref{thm-comparison}]
For simplicity, we write $r_o$ as $r$. When $x\in B_o'(R)\setminus\cut(o)$ where $\cut(o)$ means the cut-locus of $o$, let $\gamma:[0,l]\to B_o(R)$ be the normal minimal geodesic joining $o$ to $x$. Let $e_1,e_2,\cdots,e_n$ be a local unitary frame along $\gamma$ with $$e_1=\frac1{\sqrt 2}(\nabla r-\ii J\nabla r)=\mathbf U(\gamma').$$
Then,
\begin{equation}
r_1=r_{\bar 1}=\frac{1}{\sqrt 2}\ \mbox{and}\ r_\a=r_\ba=0\ \mbox{for $\alpha>1$.}
\end{equation}
Let $f(t)=L[r](\gamma(t))$. It is clear by \eqref{eq-L-f-L} that
$$L[r]=2r_{1\bar 1}=4r_{i\bar j}r_{\bi} r_j,$$
where $r_{i\bar j}=\Hess_D(f)(e_i,\ol{e_j})$. So, by equation (3.27) in the proof of Theorem 3.2 in \cite{Yu2},
\begin{equation}\label{eq-f}
\begin{split}
f'=&-f^2-R_{1\bar 1 1\bar 1}-4\sum_{i=2}^n\left(|r_{1\bi}|^2-2\mbox{Re}\{r_{1\bi}(\tau_{i1}^{\bar 1}+\tau_{i1}^1)/\sqrt 2\}+\frac{1}{4}|\tau_{i1}^1+\tau_{i1}^{\bar 1}|^2\right)\\
=&-f^2-\left(R_{1\bar 1 1\bar 1}-\sum_{i=2}^n|\tau_{i1}^1+\tau_{i1}^{\bar 1}|^2\right)-4\sum_{i=2}^n\left|r_{1\bar i}-\frac{1}{\sqrt 2}(\tau_{\bi\bar 1}^{ 1}+\tau_{\bi\bar1}^{\bar1})\right|^2.\\
\end{split}
\end{equation}
Hence
\begin{equation}
f'+f^2+H(\gamma')\leq 0\leq h'+h^2+H(\gamma')
\end{equation}
for $t\in (0,l]$. Note that $f(r)\sim\frac1r$ as $r\to0^+$. Thus, by comparison of Riccati equation (See \cite{Ro}),
\begin{equation}
L[r](x)=f(l)\leq h(l).
\end{equation}
Then, by direct computation,
\begin{equation}
L[v(r)](x)=(v')^2(v''+v'L[r])(x)\leq (v')^2(v''+h v')(l)\leq 0.
\end{equation}

When $x\in B'_o(R)$ is on the cut-locus of $o$, let $\gamma:[0,l]\to B_o(R)$ be a normal minimal geodesic joining $o$ to $x$. For $\delta\in [0,l]$ small enough, we denote $h_\delta(t)=h(t+\delta)$ for $t\in (0,l-\delta]$. Let
\begin{equation}
w_\delta(t)=\left(\int_0^te^{2\int_\tau^t h_\delta(s)ds}d\tau\right)^{-1}=\left(\int_0^te^{2\int_{\tau+\delta}^{t+\delta} h(s)ds}d\tau\right)^{-1}>0
\end{equation}
when $t\in (0,l-\delta]$. It is not hard to check that $w_\delta$ satisfies
\begin{equation}\label{eq-w}
w'_\delta+w_\delta^2+2h_\delta w_\delta=0
\end{equation}
on $(0,l-\delta]$ and $w_\delta(t)\sim\frac1t$ as $t\to0^+$. Moreover,
\begin{equation}
\lim_{\delta\to 0^+}w_\delta(l-\delta)=\left(\lim_{\delta\to 0^+}\int_\delta^{l}e^{2\int_{\tau}^{l} h(s)ds}d\tau\right)^{-1}=0
\end{equation}
since $h(t)\sim\frac1t$ as $t\to 0^+$. For each $\e>0$, let $\delta\in (0,l)$ be small enough such that
\begin{equation}\label{eq-w-epsilon}
w_\delta(l-\delta)<\e.
\end{equation}
Let $p=\gamma(\delta)$. Then, $x$ is not a cut point of $p$. By the triangle inequality, we have
\begin{equation}
r\leq r_p+\delta\ \mbox{and}\ r(x)=r_p(x)+\delta.
\end{equation}
Moreover,  let $f_\delta(t)=L[r_p](\gamma(t+\delta))$, by \eqref{eq-f}
\begin{equation}
f'_\delta(t)+f_\delta^2(t)+H(\gamma'(t+\delta))\leq 0
\end{equation}
for $t\in [0,l-\delta]$. On the other hand, by \eqref{eq-w},
\begin{equation}\label{eq-f-delta}
\begin{split}
&(h_\delta+w_\delta)'(t)+(h_\delta+w_\delta)^2(t)+H(\gamma'(t+\delta))\\
=&h'(t+\delta)+h^2(t+\delta)+H(\gamma'(t+\delta))\\
\geq&0
\end{split}
\end{equation}
for any $t\in (0,l-\delta]$. So, by comparison of Ricatti equation as before, \eqref{eq-f-delta} and \eqref{eq-w-epsilon}, we know that
\begin{equation}
L[r_p](x)=f_\delta(l-\delta)\leq h_\delta(l-\delta)+w_\delta(l-\delta)\leq h(l)+\e
\end{equation}
for $t\in [0,l-\delta]$. Therefore,
\begin{equation}
\begin{split}
&L[v(r_p+\delta)](x)\\
=&v'(l)^2(v''(l)+v'(l)L[r_p](x))\\
\leq&(v'(l))^2(v''(l)+v'(l)(h(l)+\e))\\
=&(v'(l))^3\e
\end{split}
\end{equation}
This proves that $L[r_o](x)\leq 0$ in the sense of barrier.

When the holomorphic sectional curvature in $B_o(R)$ is not less than $K$, let
\begin{equation}
h(t)=\left\{\begin{array}{cl}\sqrt K\cot(\sqrt K t)&K>0\\
\frac1t&K=0\\
\sqrt{-K}\coth(\sqrt{-K}t)&K<0.
\end{array}\right.
\end{equation}
Then, $h$ satisfies \eqref{eq-Riccati-h}. Finally, by solving the equation $v''+hv'=0$, we get \eqref{eq-comparison-K}. This completes the proof of the theorem.
\end{proof}

\section{Maximum principle and three circle theorem}
In this section, motivated by Calabi's work \cite{Ca}, we obtain a general maximum principle and three circle theorem, and apply them to prove Theorem \ref{thm-max-principle-AH} and Theorem \ref{thm-three-circle-AH}.
\begin{thm}\label{thm-max-principle}
Let $(M^n,g)$ be a Riemannian manifold, $T_1,T_2,\cdots, T_m$ be $(1,1)$-tensors on $M$, and $Q$ be a $(0,3)$ tensor on $M$. For any smooth function $f$ on $M$, define
\begin{equation}
\mathscr{L}[f]=\sum_{i=1}^m\Hess(f)(T_i\nabla f,T_i\nabla f)+Q(\nabla f,\nabla f,\nabla f).
\end{equation}
Let $\Omega$ be a precompact open subset in $M$, and $u$ and $v$ be two continuous functions on $\ol\Omega$ such that for any $\e>0$ and $x\in \Omega$, there are two smooth functions $u_{x,\e}$ and $v_{x,\e}$ defined on some neighborhood $U_{x,\e}$ of $x$, and a positive constant $c_x$ independent of $\e$, satisfying the following properties:
\begin{enumerate}
\item $u\geq u_{x,\e}$ on $U_{x,\e}$ and $u(x)=u_{x,\e}(x)$;
\item $\mathscr{L}[u_{x,\e}](x)\geq-\e$;
\item $v\leq v_{x,\e}$ on $U_{x,\e}$ and $v(x)=v_{x,\e}(x)$;
\item $\mathscr{L}[v_{x,\e}](x)\leq \e$;
\item $\sum_{i=1}^m\vv<T_i\nabla u_{x,\e},\nabla u_{x,\e}>^2(x)+\sum_{i=1}^m\vv<T_i\nabla v_{x,\e},\nabla v_{x,\e}>^2(x)\geq c_x$;
\item $u|_{\p\Omega}\leq v|_{\p\Omega}$.
\end{enumerate}
Then, $u\leq v$ in $\Omega$.
\end{thm}
\begin{proof}We will proceed by contradiction. Suppose $u>v$ for some point in $\Omega$. For each $\delta>0$ small enough, let $u_\delta=a_\delta\ln(e^u+\delta)-b_\delta$ where $a_\delta=\frac{1}{1+\delta}$ and $$b_\delta=\frac{\ln(e^{-\delta\min_{\p\Omega}u}+\delta e^{-(1+\delta)\min_{\p\Omega}u})}{1+\delta}.$$
Then, $u_\delta|_{\p\Omega}\leq u|_{\p \Omega}$. Similarly, let $v_\delta=A_\delta\ln(e^v-\delta)+B_\delta$ where $A_\delta=\frac{1}{1-\delta}$ and
$$B_\delta=-\frac{\ln(e^{\delta\min_{\p\Omega}v}-\delta e^{-(1-\delta)\min_{\p\Omega}v})}{1-\delta}.$$
Then, $v_\delta|_{\p\Omega}\geq v|_{\p\Omega}$. Note that $u_\delta\to u$ and $v_\delta\to v$ as $\delta\to0^+$. So, we can fix $\delta>0$  small enough, such that $\max_{\ol\Omega}(u_\delta-v_\delta)>0$. Let $x_\delta\in\Omega$ be a maximum point of $u_\delta-v_\delta$. For each $\e>0$, let
\begin{equation}
u_{\delta,\e}=a_\delta\ln(e^{u_{x_\delta,\e}}+\delta)-b_\delta
\end{equation}
and
\begin{equation}
v_{\delta,\e}=A_\delta\ln(e^{v_{x_\delta,\e}}-\delta)+B_\delta.
\end{equation}
Then, it is clear that $x_\delta$ is a maximum point of $u_{\delta,\e}-v_{\delta,\e}$ in the neighborhood $U_{x_\delta,\e}$ of $x_\delta$ by the assumptions (1) and (3). So,
\begin{equation}
\nabla u_{\delta,\e}(x_\delta)=\nabla v_{\delta,\e}(x_\delta)
\end{equation}
and
\begin{equation}
\Hess(u_{\delta,\e})(x_\delta)\leq \Hess(v_{\delta,\e})(x_\delta).
\end{equation}
Thus,
\begin{equation}\label{eq-Lu-leq-Lv}
\mathscr{L}[u_{\delta,\e}](x_\delta)\leq \mathscr{L}[v_{\delta,\e}](x_\delta).
\end{equation}
Moreover, by direct computation and the assumptions (2) and (4),
\begin{equation}
\begin{split}
\mathscr{L}[u_{\delta,\e}](x_\delta)=&\left(\frac{a_\delta e^{u(x_\delta)}}{e^{u(x_\delta)}+\delta}\right)^3\left(\mathscr{L}[u_{x_\delta,\e}](x_\delta)+\frac{\delta}{e^{u(x_\delta)}+\delta}\sum_{i=1}^m\vv<T_i\nabla u_{x_\delta,\e},\nabla u_{x_\delta,\e}>^2(x_\delta)\right)\\
\geq&\left(\frac{a_\delta e^{u(x_\delta)}}{e^{u(x_\delta)}+\delta}\right)^3\left(-\e+\frac{\delta}{e^{u(x_\delta)}+\delta}\sum_{i=1}^m\vv<T_i\nabla u_{x_\delta,\e},\nabla u_{x_\delta,\e}>^2(x_\delta)\right)\\
\end{split}
\end{equation}
and
\begin{equation}
\begin{split}
\mathscr{L}[v_{\delta,\e}](x_\delta)=&\left(\frac{A_\delta e^{v(x_\delta)}}{e^{v(x_\delta)}-\delta}\right)^3\left(\mathscr{L}[v_{x_\delta,\e}](x_\delta)-\frac{\delta}{e^{v(x_\delta)}-\delta}\sum_{i=1}^m\vv<T_i\nabla v_{x_\delta,\e},\nabla v_{x_\delta,\e}>^2(x_\delta)\right)\\
\leq&\left(\frac{A_\delta e^{v(x_\delta)}}{e^{v(x_\delta)}-\delta}\right)^3\left(\e-\frac{\delta}{e^{v(x_\delta)}-\delta}\sum_{i=1}^m\vv<T_i\nabla v_{x_\delta,\e},\nabla v_{x_\delta,\e}>^2\right).\\
\end{split}
\end{equation}
Then, by assumption (5),
\begin{equation}
\begin{split}
&\mathscr{L}[u_{\delta,\e}](x_\delta)-\mathscr{L}[v_{\delta,\e}](x_\delta)\\
\geq&-\left[\left(\frac{a_\delta e^{u(x_\delta)}}{e^{u(x_\delta)}+\delta}\right)^3+\left(\frac{A_\delta e^{v(x_\delta)}}{e^{v(x_\delta)}-\delta}\right)^3\right]\e+\min\left\{\frac{\delta a_\delta^3 e^{3u(x_\delta)}}{(e^{u(x_\delta)}+\delta)^4},\frac{\delta A_\delta^3 e^{3v(x_\delta)}}{(e^{v(x_\delta)}-\delta)^4}\right\}\\
&\times\left(\sum_{i=1}^m\vv<T_i\nabla u_{x_\delta,\e},\nabla u_{x_\delta,\e}>^2(x_\delta)+\sum_{i=1}^m\vv<T_i\nabla v_{x_\delta,\e},\nabla v_{x_\delta,\e}>^2(x_\delta)\right)\\
\geq&-\left[\left(\frac{a_\delta e^{u(x_\delta)}}{e^{u(x_\delta)}+\delta}\right)^3+\left(\frac{A_\delta e^{v(x_\delta)}}{e^{v(x_\delta)}-\delta}\right)^3\right]\e+\min\left\{\frac{\delta a_\delta^3 e^{3u(x_\delta)}}{(e^{u(x_\delta)}+\delta)^4},\frac{\delta A_\delta^3 e^{3v(x_\delta)}}{(e^{v(x_\delta)}-\delta)^4}\right\}c_{x_\delta}\\
>&0
\end{split}
\end{equation}
when $\epsilon>0$ is small enough. This contradicts \eqref{eq-Lu-leq-Lv} and completes the proof of the theorem.
\end{proof}
\begin{rem}
Assumption (5) in Theorem \ref{thm-max-principle} can be viewed as a nondegenerate condition on $T_1,T_2,\cdots T_m$. It is necessary for the maximum principle. Otherwise, consider the case $T_1=T_2=\cdots=T_m=0$. It is clear that the maximum principle is not true in this case because the crucial second order terms in $\mathscr{L}$ vanishes. In our application, $\mathscr{L}=L$ with $T_1=I$ and $T_2=J$ and $v$ is a strictly increasing function of $r$. So, assumption (5) in Theorem \ref{thm-max-principle} will be automatically satisfied.
\end{rem}
By the maximum principle above, we have the following generalized three circle theorem.
\begin{thm}\label{thm-general-three-circle}
Let $(M^n,g)$ be a Riemannian manifold, $T_1,T_2,\cdots, T_m$ be $(1,1)$-tensors on $M$, and $Q$ be a $(0,3)$ tensor on $M$. For any smooth function $f$ on $M$, define
\begin{equation}
\mathscr{L}[f]=\sum_{i=1}^m\Hess(f)(T_i\nabla f,T_i\nabla f)+Q(\nabla f,\nabla f,\nabla f).
\end{equation}
Let $\Omega$ be an open subset in $M$, and $u$ and $v$ be two continuous function on $\Omega$ such that for any $\e>0$ and $x\in \Omega$, there are two smooth functions $u_{x,\e}$ and $v_{x,\e}$ defined on some neighborhood $U_{x,\e}$ of $x$, and a positive constant $c_x$ independent of $\e$, satisfying the following properties:
\begin{enumerate}
\item $u\geq u_{x,\e}$ on $U_{x,\e}$ and $u(x)=u_{x,\e}(x)$;
\item $\mathscr{L}[u_{x,\e}](x)\geq-\e$;
\item $v\leq v_{x,\e}$ on $U_{x,\e}$ and $v(x)=v_{x,\e}(x)$;
\item $\mathscr{L}[v_{x,\e}](x)\leq \e$;
\item $\sum_{i=1}^m\vv<T_i\nabla u_{x,\e},\nabla u_{x,\e}>^2(x)+\sum_{i=1}^m\vv<T_i\nabla v_{x,\e},\nabla v_{x,\e}>^2(x)\geq c_x$;
\item $v:\Omega\to (\inf_\Omega v,\sup_\Omega v)$ is proper;
\item $M_v(u,t):=\max_{x\in S_v(t)}u(x)$ is increasing on $t\in (\inf_\Omega v,\sup_\Omega v)$ where $S_v(t)=\{x\in\Omega\ |\ v(x)=t\}.$
\end{enumerate}
Then, $M_v(u,t)$ is a convex function of $t$ in $(\inf_\Omega v,\sup_\Omega v)$.
\end{thm}
\begin{proof}
For any $\inf_\Omega u<t_1<t_2<t_3<\sup_\Omega u$.
If $M_v(u,t_3)=M_v(u,t_1)$, by that $M_v(u,t)$ is increasing on $t$, we know that $M_v(u,t)$ is a constant when $t\in [t_1,t_3]$. So,
\begin{equation}
M_v(u,t_2)=\frac{t_3-t_2}{t_3-t_1}M_v(u,t_1)+\frac{t_2-t_1}{t_3-t_1}M_v(u,t_3).
\end{equation}
When $M_v(u,t_3)>M_v(u,t_1)$. Let
\begin{equation}
\tilde v=\frac{M_v(u,t_3)-M_v(u,t_1)}{t_3-t_1}v+\frac{M_v(u,t_1)t_3-M_v(u,t_3)t_1}{t_3-t_1}.
\end{equation}
Then  $\tilde v\geq u$ on $S_v(t_1)\cup S_v(t_3)$. So, by Theorem \ref{thm-max-principle},
\begin{equation}
\tilde v\geq u
\end{equation}
in $A_v(t_1,t_3):=\{x\in\Omega\ |\ t_1<v(x)<t_3\}$. Then, for any $x\in S_v(t_2)$,
\begin{equation}
u(x)\leq \frac{M_v(u,t_3)-M_v(u,t_1)}{t_3-t_1}t_2+\frac{M_v(u,t_1)t_3-M_v(u,t_3)t_1}{t_3-t_1}.
\end{equation}
and hence
\begin{equation}
M_v(u,t_2)\leq \frac{t_3-t_2}{t_3-t_1}M_v(u,t_1)+\frac{t_2-t_1}{t_3-t_1}M_v(u,t_3).
\end{equation}
This completes the proof of the theorem.
\end{proof}
As an application of the general maximum principle  and three circle theorem. We come to prove Theorem \ref{thm-max-principle-AH} and Theorem \ref{thm-three-circle-AH}, a maximum principle and a three circle theorem for functions satisfying $L[u]\geq 0$ on almost Hermitian manifolds.

\begin{proof}[Proof of Theorem \ref{thm-max-principle-AH}]
Let $o\in M$ be such that $r(o,\ol\Omega)>1$ and suppose that $\ol\Omega\subset B_o(R)$. Let $-K$ with $K>0$ be a lower bound for the holomorphic sectional curvature of $M$ in $B_o(R)$. Then, for any given $\delta>0$, by Theorem \ref{thm-comparison}, we know that $u$ and $v=\max_{\p\Omega}u+\delta(\log\tanh( \sqrt Kr_o/2)-\log\tanh(\sqrt K/2))$ satisfies the assumptions of Theorem \ref{thm-max-principle} with $\mathscr L=L$ on $\Omega$. So
\begin{equation}
u\leq \max_{\p\Omega}u+\delta\left(\log\tanh( \sqrt Kr_o/2)-\log\tanh(\sqrt K/2)\right)
\end{equation}
for any $\delta>0$. Letting $\delta\to 0^+$ in the last inequality, we get the conclusion.
\end{proof}

\begin{proof}[Proof of Theorem \ref{thm-three-circle-AH}]
By Theorem \ref{thm-comparison} and Theorem \ref{thm-max-principle-AH}, we know that $u$ and $v(r_o)$ satisfy the assumptions of Theorem \ref{thm-general-three-circle} for $\mathscr L=L$ and $\Omega=B'_o(R)$. Hence, by Theorem \ref{thm-general-three-circle}, we get the general conclusion.

The special case for pluri-subharmonic functions comes from the general conclusion by using the fact that pluri-subharmonic functions satisfy $L[u]\geq 0$ automatically, by Lemma \ref{lem-ddbar-f} and \eqref{eq-L-f-L}.

The special case for holomorphic functions comes from the special case for pluri-subharmonic functions by using the fact that $\log (|f_1|^2+|f_2|^2+\cdots+|f_k|^2+\e)$ is pluri-subharmonic by Lemma \ref{lem-holo-psh}, and letting $\e\to 0^+$.
\end{proof}
\section{Almost Hermitian manifolds with slightly negative holomorphic sectional curvature}

In this section, we prove the main results for almost Hermitian manifolds by applying the framework that set up in the last section. We first need the following lemmas on asymptotic behaviors of solutions for the ordinary differential equations related to the equations appeared in Theorem \ref{thm-comparison}.
\begin{lem}\label{lem-second-ODE}
Let $q(t)$ be a nonnegative continuous function on $[0,+\infty)$ and $u\in C^2([0,+\infty))$ be the solution of
$$u''-qu=0$$
with $u(0)=0$ and $u'(0)=1$. Then,
\begin{enumerate}
\item $u'$ is increasing;
\item $u'(+\infty):=\D\lim_{t\to +\infty} u'(t)<+\infty$ if and only if $\D\int_0^{+\infty}tq(t)dt<+\infty$;
\item when $\D\int_0^{+\infty}tq(t)dt<+\infty$, $u'(+\infty)\leq I(q).$
\end{enumerate}
\end{lem}
\begin{proof}
It is clear that $u>0$ when $t>0$ and hence $u''=qu\geq0$ on $[0,+\infty)$. So, $u'(t)$ is increasing and
$$u'(t)\geq u'(0)=1.$$
Thus $u(t)\geq t$. If $u'(+\infty)<+\infty$, by that
\begin{equation}
u'(+\infty)-1=\int_0^{+\infty}u''(s)ds=\int_0^{+\infty}q(s)u(s)ds\geq \int_0^{+\infty}sq(s)ds,
\end{equation}
one has
\begin{equation}
\int_0^{+\infty}tq(t)dt<+\infty.
\end{equation}
On the other hand, when $\int_0^{+\infty}tq(t)dt<+\infty$, note that
\begin{equation}
\begin{split}
u(t)=&\int_0^tu'(s)ds\\
=&t+\int_0^t\int_0^su''(\tau)d\tau ds\\
=&t+\int_0^t\int_0^sq(\tau)u(\tau)d\tau ds\\
=&t+\int_0^t(t-\tau)q(\tau)u(\tau)d\tau \\
\leq&t\left(1+\int_0^tq(\tau)u(\tau)d\tau\right). \\
\end{split}
\end{equation}
Let $F(t)=\int_0^tq(\tau)u(\tau)d\tau$. Then, by the last inequality,
\begin{equation}
\begin{split}
F'(t)=q(t)u(t)\leq tq(t)\left(1+F(t)\right).
\end{split}
\end{equation}
So,
\begin{equation}
\left[e^{-\int_0^ts q(s)ds}F(t)\right]'=e^{-\int_0^ts q(s)ds}[F'(t)-tq(t)F(t)]\leq e^{-\int_0^ts q(s)ds}tq(t).
\end{equation}
Thus,
\begin{equation}
F(t)\leq e^{\int_0^tsq(s)ds}\int_0^te^{-\int_0^\tau s q(s)ds}\tau q(\tau)d\tau=e^{\int_0^tsq(s)ds}-1.
\end{equation}
for $t\geq 0$. Then,
\begin{equation}
u'(t)=1+\int_0^tu''(s)ds=1+\int_0^tq(s)u(s)ds=1+F(t)\leq I(q).
\end{equation}
This completes the proof.
\end{proof}
\begin{lem}\label{lem-Riccati}
Let $q(t)$ be a nonnegative continuous function on $[0,+\infty)$ such that $\int_0^{+\infty}tq(t)dt<+\infty$, and $h$ be the solution of
\begin{equation}
h'+h^2-q(t)=0
\end{equation}
with $\lim_{t\to 0^+}th(t)=1$. Let $v\in C^2((0,+\infty))$ the solution of
\begin{equation}
v''+hv'=0
\end{equation}
with $v'>0$ and $\D\lim_{t\to 0^+}\frac{v(t)}{\log t}=1$. Then, $\D\lim_{t\to+\infty}\frac{v(t)}{\log t}$ exists and
 $$\frac1{I(q)}\leq \lim_{t\to+\infty}\frac{v(t)}{\log t}\leq 1.$$
 Moreover,
\begin{equation}\label{eq-v-t>1}
\log t+v(1)\geq v(t)\geq \frac{1}{I(q)}\log t+v(1),\ \forall t>1.
\end{equation}

\end{lem}
\begin{proof}
Let $u$ be the function in Lemma \ref{lem-second-ODE}. It is not hard to check that
$$h=\frac{u'}{u}.$$
Substituting this into the equation of $v$, we  have
\begin{equation}
(uv')'=0.
\end{equation}
Thus $v'=\frac{c}{u}$ for some positive constant $c$. By L'Hospital's rule,
$$1=\lim_{t\to 0^+}\frac{v(t)}{\log t}=\lim_{t\to 0^+}tv'(t)=\lim_{t\to 0^+}\frac{ct}{u(t)}=c.$$
By L'Hospital's rule again,
\begin{equation}
\lim_{t\to \infty}\frac{v(t)}{\log t}=\lim_{t\to +\infty}{tv'(t)}=\lim_{t\to +\infty}\frac{t}{u(t)}=\lim_{t\to +\infty}\frac{1}{u'(t)}=\frac1{u'(+\infty)}\geq \frac{1}{I(q)}.
\end{equation}
Moreover, because $u'$ is increasing, $u'(+\infty)\geq u'(0)=1$. Thus
\begin{equation}
\lim_{t\to +\infty}\frac{v(t)}{\log t}=\frac1{u'(+\infty)}\leq 1.
\end{equation}
Finally, for any $t>1$,
\begin{equation}
\begin{split}
\log t=\int_1^t\frac{1}{s}ds\geq v(t)-v(1)=\int_1^t\frac1uds\geq \frac1{I(q)}\int_1^t\frac{1}{s}ds=\frac1{I(q)}\log t.
\end{split}
\end{equation}
This proves \eqref{eq-v-t>1} and completes the proof of the lemma.
\end{proof}
 We are now ready to prove the Liouville-type result in Corollary \ref{cor-Liouville-psh-2} by using Theorem \ref{thm-three-circle-AH}.
\begin{proof}[Proof of Corollary \ref{cor-Liouville-psh-2}] Note that for any $p\in M$, by Theorem \ref{thm-max-principle-AH},
\begin{equation}
M_p(u,r)\leq M_o(u,r+r(o,p)).
\end{equation}
So,
\begin{equation}\label{eq-sub-log-p}
\liminf_{r\to+\infty}\frac{M_p(u,r)}{\log r}\leq \liminf_{r\to+\infty}\frac{M_o(u,r+r(o,p))}{\log r}\leq 0.
\end{equation}
Moreover, by Theorem \ref{thm-three-circle-AH}, $M_p(u,r)$ is convex with respect to $v_p(r)$ where $v_p$ is the function associated to $q_p$. So, for any $0<r_1<r_2<r_3$,
\begin{equation}
M_p(u,r_2)\leq \frac{v_p(r_3)-v_p(r_2)}{v_p(r_3)-v_p(r_1)}M_p(u,r_1)+\frac{v_p(r_2)-v_p(r_1)}{v_p(r_3)-v_p(r_1) }M_p(u,r_3)
\end{equation}
Taking $\liminf_{r_3\to+\infty}$ in the last inequality, by \eqref{eq-sub-log-p} and Lemma \ref{lem-Riccati}, we get
\begin{equation}
M_p(u,r_2)\leq M_p(u,r_1).
\end{equation}
Then, by Theorem  \ref{thm-max-principle-AH},
$$M_p(u,r_1)=M_p(u,r_2)$$
for any $0<r_1<r_2$. Letting $r_1\to 0^+$ in the last equation, we get
\begin{equation}
u(p)=M_p(u,r_2)
\end{equation}
for any $r_2>0$. So $u(p)=\max_M u$. Note that $p$ is arbitrary. Thus $u$ is a constant function.

By Lemma \ref{lem-ddbar-f} and \eqref{eq-L-f-L}, we know that a pluri-subharmonic function $u$ will automatically satisfy $L[u]\geq 0$. So, we have the same Liouville property for pluri-subharmonic functions.
\end{proof}

We next come to prove Theorem \ref{thm-mono-2}.
\begin{proof}[Proof of Theorem \ref{thm-mono-2}](1) It is clear by \eqref{eq-poly} that if $f\in \mathcal P(M)$, then
\begin{equation}
\liminf_{r\to +\infty}\frac{\log M_o(|f|,r)}{\log r}\leq \limsup_{r\to +\infty}\frac{\log M_o(|f|,r)}{\log r}<+\infty.
\end{equation}
Conversely, suppose that $\liminf_{r\to +\infty}\frac{\log M_o(|f|,r)}{\log r}=\lambda$. Then, by Theorem \ref{thm-three-circle-AH}, $\log M_o(|f|,r)-\frac{\lambda}{\alpha} v(r)$ is a convex function of $v(r)$, where $\alpha=\D\lim_{r\to+\infty}\frac{v(r)}{\log r}$. Then,
\begin{equation}\label{eq-convex-2}
\begin{split}
&\log M_o(|f|,r_2)-\frac{\lambda}{\alpha} v(r_2)\\
\leq& \frac{v(r_2)-v(r_1)}{v(r_3)-v(r_1)}\left(\log M_o(|f|,r_3)-\frac{\lambda}{\alpha} v(r_3)\right)+\frac{v(r_3)-v(r_2)}{v(r_3)-v(r_1)}\left(\log M_o(|f|,r_1)-\frac{\lambda}{\alpha} v(r_1)\right)
\end{split}
\end{equation}
for any $0<r_1<r_2<r_3$. Taking $\liminf_{r_3\to+\infty}$ in the last inequality, we know that $\log M_o(|f|,r)-\frac{\lambda}{\alpha}v(r)$ is decreasing. So,
\begin{equation}
\log M_o(|f|,r)\leq \log M_o(|f|,1)-\frac{\lambda}{\alpha}v(1)+\frac{\lambda}{\alpha}v(r)
\end{equation}
for any $r>1$, and hence
\begin{equation}
\limsup_{r\to+\infty}\frac{\log M_o(|f|,r)}{\log r}\leq \lambda.
\end{equation}
Thus $f\in \mathcal O_\lambda(M)\subset \mathcal P(M)$.

(2) and (3) is clear from the proof of (1) by noting that $\alpha\geq\frac1{I(q)}$.

(4) By taking $\lambda=\ord_o(f)$ and $\alpha=1$ in \eqref{eq-convex-2}, letting $r_1\to 0^+$, and noting that $$\lim_{r\to 0^+}\frac{\log M_o(|f|,r)}{v(r)}=\lim_{r\to 0^+}\frac{\log M_o(|f|,r)}{\log r}=\ord_o(f),$$ one  gets  the conclusion.

(5) is a direct corollary of (3) and (4).
\end{proof}
\begin{rem}
For a holomorphic function $f$ on a domain $\Omega$, by Lemma \ref{lem-ddbar-f}, we know that $\Delta_Df=0$ where $\Delta_D=\tr_g\Hess_D$ is the Laplacian operator with respect to $D$. Then, by the unique continuation theorem of Aronszajn \cite{Ar}, $\ord_p(f)$ is finite for any $p\in\Omega$ when $f\not\equiv 0$.  
\end{rem}

By (4) of Theorem \ref{thm-mono-2}, we can prove Corollary \ref{cor-Liouville-holo-2}, a Liouville theorem of Cheng-type for holomorphic functions.
\begin{proof}[Proof of Corollary \ref{cor-Liouville-holo-2}]
Suppose $f$ is nonconstant. Let $h=f-f(o)$ and $k=\ord_o(h)\geq 1$. By (4) of Theorem \ref{thm-mono-2}, $\frac{M_o(|h|,r)}{e^{kv(r)}}$ is an increasing function. So,
\begin{equation}
\frac{|f(o)|+M_o(|f|,r_2)}{e^{kv(r_2)}}\geq\frac{M_o(|h|,r_2)}{e^{kv(r_2)}}\geq\frac{ M_o(|h|,r_1)}{e^{kv(r_1)}}
\end{equation}
for all $r_2>r_1>0$. Letting $r_2\to+\infty$ in the last inequality, by Lemma \ref{lem-Riccati} and \eqref{eq-sub-linear}, we have $M_o(|h|,r_1)=0$ for any $r_1>0$. Therefore $h\equiv0$ and $f\equiv f(o)$ is a constant. This is a contradiction.
\end{proof}

By Corollary \ref{cor-Liouville-holo-2}, we can prove Corollary \ref{cor-fund-holo-2}, an analogue of fundamental theorem for algebra on almost Hermitian manifolds with slightly negative holomorphic sectional curvature.
\begin{proof}[Proof of Corollary \ref{cor-fund-holo-2}]
 Let $f$ be a nonconstant holomorphic function of polynomial growth on $M$ such that $f\neq 0$ everywhere.  Consider the universal covering $\pi:\C\to \C^*$ with $\pi(z)=e^z$. Because $M$ is simply connected, $f:M\to \C^*$ has a lift $h:M\to \C$ along $\pi$ which is clearly still holomorphic. That is, $f=e^{h}$ for some holomorphic function $h$ on $M$. Let $N>2I(q)\deg f$ be a natural number  and $\zeta=e^\frac{h}{N}$. Then, $\zeta$ is a holomorphic function on $M$ such that $f=\zeta^N$. Thus, $\zeta$ is a holomorphic function on $M$ satisfying \eqref{eq-sub-linear}, because $N>2I(q)\deg f$. By Corollary \ref{cor-Liouville-holo-2}, $\zeta$ is constant and hence $f$ is constant. This is a contradiction.
\end{proof}
At the end of this section, we come to prove Theorem \ref{thm-dim}. Before proving Theorem \ref{thm-dim}, we need the following lemma.
\begin{lem}\label{lem-vanishing-order}
Let $(M^{2n}, J)$ be an almost complex manifold of real dimension $2n$ and $o\in M$. Let $(U,z)$ a local complex coordinate at $o$ such that $\frac{\p}{\p z^i}\big|_o\in T^{1,0}_oM$ for $i=1,2,\cdots,n$, and $f$ be a holomorphic function on $U$ such that
$$\frac{\p^{|\alpha|} f}{\p z^\alpha}(o)=0$$
 for any multi-index $\alpha$ with $|\alpha|\leq m$. Then, $\ord_o(f)\geq m+1$.
\end{lem}
\begin{proof}
Suppose that
$$J\frac{\p}{\p z^i}=J_i^j\frac{\p}{\p z^j}+J_i^{\bar j}\frac{\p}{\p z^\bj}$$
and
$$J\frac{\p}{\p z^{\bi}}=J_{\bi}^j\frac{\p}{\p z^j}+J_{\bi}^{\bj}\frac{\p}{\p z^{\bj}}.$$
It is clear that $\ol{J_i^j}=J_\bi^\bj$ and $\ol{J_i^\bj}=J_\bi^j$ and moreover $J_i^j(o)=\ii\delta_i^j$ and $J_i^\bj(o)=0$ since $\frac{\p}{\p z^i}\big|_o$ is a $(1,0)$-vector. Note that
\begin{equation}
\begin{split}
\dbar f=&(df)^{(0,1)}\\
=&\frac{1}{2}(df+\ii Jdf)\\
=&\frac12\left(\frac{\p f}{\p z^i}dz^i+\frac{\p f}{\p z^\bi}dz^\bi+\ii\left(\frac{\p f}{\p z^j}Jdz^j+\frac{\p f}{\p z^\bj}Jdz^\bj\right)\right)\\
=&\frac12\left(\left(\frac{\p f}{\p z^i}+\ii\frac{\p f}{\p z^j}J_i^j+\ii\frac{\p f}{\p z^\bj}J_i^\bj\right)dz^i+\left(\frac{\p f}{\p z^\bi}+\ii\frac{\p f}{\p z^j}J_\bi^j+\ii\frac{\p f}{\p z^\bj}J_\bi^\bj\right)dz^\bi\right).
\end{split}
\end{equation}
So,
\begin{equation}
\frac{\p f}{\p z^\bi}+\ii\frac{\p f}{\p z^j}J_\bi^j+\ii\frac{\p f}{\p z^\bj}J_\bi^\bj=0.
\end{equation}
This implies that, in a neighborhood of $o$,
\begin{equation}\label{eq-relation-d-dbar}
\frac{\p f}{\p \bar z}=A(z)\frac{\p f}{\p z}
\end{equation}
where $\frac{\p f}{\p \bar z}=(\frac{\p f}{\p z^{\bar 1}},\cdots, \frac{\p f}{\p z^{\bar n}})^T$ and $\frac{\p f}{\p z}=(\frac{\p f}{\p z^{ 1}},\cdots, \frac{\p f}{\p z^{ n}})^T.$ Here $A$ is the product of the inverse matrix of $(\delta_i^j+\ii J_\bi^\bj)_{i,j=1,2,\cdots,n}$ and $-(\ii J_\bi^j)_{i,j=1,2,\cdots,n}$.

We next come to prove that
\begin{equation}\label{eq-claim}
\frac{\p^{|\alpha|+|\beta|}f}{\p z^\alpha\p\bar z^\beta}(o)=0
\end{equation}
for any multi-index $\alpha$ and $\beta$ with $|\alpha|+|\beta|\leq m$ by induction on $|\beta|$. The claim is clearly true when $|\beta|=0$ by assumption. Suppose the claim is true for $|\beta|\leq k<m$. When $|\beta|=k+1$, let $\gamma$ be a multi-index with $|\gamma|=k$ such that $\beta_i=\gamma_i+1$ for some fixed $i$. Then, by \eqref{eq-relation-d-dbar}, \begin{equation}
\frac{\p f}{\p z^\bi}=\sum_{j=1}^nA_{ij}(z)\frac{\p f}{\p z^j}.
\end{equation}
So,
\begin{equation}
\frac{\p^{|\alpha|+|\beta|}f}{\p z^\alpha\p \bar z^\beta}(o)=\sum_{j=1}^n\frac{\p^{|\alpha|+|\gamma|}}{\p z^\alpha\p \bar z^\gamma}\left(A_{ij}(z)\frac{\p f}{\p z^j}\right)(o)=0
\end{equation}
by the induction hypothesis.

By \eqref{eq-claim}, we know that the partial derivatives of $f$ vanishes up to order $m$. Thus $\ord_o(f)\geq m+1$.
\end{proof}

We are now ready to prove Theorem \ref{thm-dim}.
\begin{proof}[Proof of Theorem \ref{thm-dim}]
Let $(z^1,z^2,\cdots,z^n)$ be a local complex coordinate at $o$ such that $z^i(o)=0$ and $\frac{\p}{\p z^i}\big|_o$ is a $(1,0)$-vector for $i=1,2,\cdots,n$.
Let $$\Phi:\mathcal O_d(M)\to \mathcal O_{[I(q)d]}(\C^n)$$ be such that $$\Phi(f)=\sum_{|\alpha|\leq [I(q)d]}\frac{1}{\alpha!}\frac{\p^{|\alpha|} f}{\p z^\alpha}(o)z^\alpha.$$
If $\Phi$ is not injective, then there is a nonzero $f\in \mathcal O_d(M)$ such that $\Phi(f)=0$. By Lemma \ref{lem-vanishing-order},
 $$\ord_o(f)\geq[I(q)d]+1>I(q)d\geq I(q)\deg f.$$
 This contradicts the vanishing order estimate in (5) of Theorem \ref{thm-mono-2}. So, we have proved the dimension estimate \eqref{eq-dim-g} and its special case \eqref{eq-dim-sharp}.

When the holomorphic sectional curvature of $M$ is nonnegative and the equality of \eqref{eq-dim-sharp} holds for some positive integer $d$, the map $\Phi:\mathcal O_d(M)\to \mathcal O_d(\C^n)$ defined above is an isomorphism. So, there are $f_1,f_2,\cdots,f_n\in\mathcal O_d(M)$ such that $\frac{\p f_j}{\p z^i}(o)=\delta_{ij}$. Then $f_1,f_2,\cdots,f_n$ form a local holomorphic coordinate at $o$. Because the point $o$ can be arbitrarily chosen, the complex structure $J$ is integrable. Thus $\tau_{ij}^\bk=0$ for any $i,j,k=1,2,\cdots,n$.

For each $\xi\in T^{1,0}_oM$ with $\|\xi\|=1$, let $(z^1,z^2,\cdots,z^n)$ be a holomorphic coordinate at $o$ with $z(o)=0$, $\frac{\p}{\p z^1}\big|_o=\xi$ and $g_{i\bar j}(o)=\delta_{ij}$. By that $\Phi$ is an isomorphism, there is a holomorphic function $f\in \mathcal O_d(M)$ such that
\begin{equation}\label{eq-expansion-f}
f(z)=(z^1)^d+O(r(z)^{d+1})
\end{equation}
Because $\ord_o(f)=d\geq \deg f$, by Theorem \ref{thm-mono-2}, $\ord_o(f)=\deg f=d$ and
\begin{equation}
\log M_o(|f|,r)=d\log r+c
\end{equation}
for some constant $c$. For each $\e>0$, let $z_\e\in S_o(\e)$ attain the maximal modulus of $f$ on $S_o(\e)$. Let
$$F(x)=\log|f|(x)-d\log r(x)-c$$
where $r(x)=r_o(x)$. Then, $F\leq 0$ and $F(z_\e)=0$. Noting that $f(z_\e)\neq 0$ and $L[\log|f|](z_\e)=0$ since $f$ is holomorphic. We have
\begin{equation}
\nabla\log|f|(z_\e)=d\nabla\log r(z_\e)
\end{equation}
and
\begin{equation}
L[\log r](z_\e)\geq 0
\end{equation}
because $z_\e$ is a maximum point for $F$. Thus, by Theorem  \ref{thm-comparison},
\begin{equation}
L[\log r](z_\e)=0
\end{equation}
and hence by \eqref{eq-f} in the proof of Theorem  \ref{thm-comparison} and \eqref{eq-diff-hess},
\begin{equation}\label{eq-tau}
\frac{1}{\sqrt 2}\tau_{\bi\bar1}^{\bar 1}(z_\e)=r_{1\bi}(z_\e)=\Hess(r)(z_\e)(e_1^\e,\ol{e_i^\e})+\frac{1}{2\sqrt 2}\tau_{\bar i\bar 1}^{\bar 1}(z_\e)
\end{equation}
for $i=2,\cdots,n$. Here, at $z_\e$, the unitary frame $e_1^\e,e_2^\e,\cdots,e_n^\e$ is chosen to be such that
$$e_1^\e=\mathbf{U}(\nabla r)(z_\e)=\mathbf{U}(\nabla|f|^2)(z_\e).$$
Note that $\mathbf{U}(\nabla|f|^2)(z_\e)$ sub-converges to $\lambda\xi$ with $|\lambda|=1$ by \eqref{eq-expansion-f}. Then, by letting $\e\to0^+$ in \eqref{eq-tau}, since
\begin{equation}
\Hess(r)(z_\e)(e_1^\e,\ol{e_i^\e})\to 0
\end{equation}
as $\e\to 0^+$ for $i=2,3,\cdots,n$ (See \cite[Lemma 2.2]{TY}), we know that
\begin{equation}
\tau_{i1}^1(o)=0
\end{equation}
for $i=2,3,\cdots,n$, where the unitary frame at $o$ is a sub-convergent limit of $(e_1^\e,e_2^\e,\cdots,e_n^\e)$ as $\e\to 0^+$. This implies that
\begin{equation}
\vv<\tau(\eta,\xi),\bar\xi>=0
\end{equation}
for any $\eta\in T_{o}^{1,0}M$ with $\eta\perp\xi$. Because $\xi$ is arbitrary chosen and $\tau$ is skew symmetric, we know that
\begin{equation}
\vv<\tau(\eta,\xi),\bar \xi>=0
\end{equation}
for any $\eta,\xi\in T_{o}^{1,0}M$. Then
\begin{equation}
\vv<\tau(\eta,\xi+\zeta), \ol{\xi+\zeta}>=0
\end{equation}
and
\begin{equation}
\vv<\tau(\eta,\xi+\ii\zeta), \ol{\xi+\ii\zeta}>=0
\end{equation}
for any $\eta,\xi,\zeta\in T_o^{1,0}M$ which implies that
$$\vv<\tau(\eta,\xi),\bar\zeta>=0$$
for any $\eta,\xi,\zeta\in T_o^{1,0}M$. Thus $\tau_{ij}^k(o)=0$ for any $i,j,k=1,2,\cdots,n$. Because $o$ is arbitrary chosen, we know that $\tau_{ij}^k=0$ all over $M$. Hence $\tau$ vanishes and $(M,J,g)$ is K\"ahler. Finally, by the rigidity part of \cite[Theorem 4]{Liu}, we know that $(M,J,g)$ must be biholomorphically isometric to $\C^n$.
\end{proof}

\section{Converse of Three Circle Theorem on Hermitian manifolds}
In this section, we will discuss the converse of the three circle theorem on Hermitian manifolds and prove Theorem \ref{thm-con-1-holo}.

For completeness, we first compute the equation for geodesics in a local holomorphic coordinate. Here, for geodesics we always mean geodesics with respect to the Levi-Civita connection.
\begin{lem}
Let $(M^{2n},J,g)$ be a Hermitian manifold and $o\in M$. Let $\gamma$ a be geodesic starting at $o$ and $z=(z^1,z^2,\cdots,z^n)$ be a local holomorphic coordinate at $o$. Suppose that $z(\gamma(t))=(z^1(t),z^2(t),\cdots,z^n(t))$. Then,
\begin{equation}\label{eq-geodesic}
\frac{d^2z^i}{dt^2}+\Gamma_{jk}^i\frac{dz^j}{dt}\frac{dz^k}{dt}+g^{\bar \lambda i}\tau_{\bar j\bar\lambda}^\bk g_{l\bk}\frac{dz^{\bar j}}{dt}\frac{dz^l}{dt}=0
\end{equation}
where $\Gamma_{ij}^k$ is the Christofel symbol for the Chern connection $D$.
\end{lem}
\begin{proof}
Let $X=\nabla_{\gamma'}\gamma'-D_{\gamma'}\gamma'=-D_{\gamma'}\gamma'=X^i\frac{\p}{\p z^i}+X^\bi\frac{\p}{\p z^\bi}$. Then, by \eqref{eq-diff-connections},
\begin{equation}
\vv<X,\frac{\p}{\p z^\bj}>=\vv<\tau\left(\gamma',\frac{\p}{\p z^\bj}\right),\gamma'>.
\end{equation}
Thus, by noting that $\tau_{ij}^\bk=0$ because of the integrability of complex structure, we have
\begin{equation}
X^ig_{i\bj}=\tau_{\bi\bj}^\bk g_{l\bk}\frac{dz^\bi}{dt}\frac{dz^l}{dt}.
\end{equation}
Therefore,
\begin{equation}
X^i=g^{\bar \lambda i}\tau_{\bar j\bar\lambda}^\bk g_{l\bk}\frac{dz^{\bar j}}{dt}\frac{dz^l}{dt}.
\end{equation}
Moreover, note that
\begin{equation}
D_{\gamma'}\gamma'=\left(\frac{d^2z^i}{dt^2}+\Gamma_{jk}^i\frac{dz^j}{dt}\frac{dz^k}{dt}\right)\frac{\p}{\p z^i}+\left(\frac{d^2z^\bi}{dt^2}+\ol{\Gamma_{jk}^i}\frac{dz^\bj}{dt}\frac{dz^\bk}{dt}\right)\frac{\p}{\p z^\bi}.
\end{equation}
So, the equation for geodesics in local holomorphic coordinates is \eqref{eq-geodesic}.
\end{proof}
We next need the following lemma for the existence of good local holomorphic coordinates for Hermitian manifolds.
\begin{lem}
Let $(M^{2n},J,g)$ be a Herimitian manifold. Then, for any $o\in M$ and  $e_1,e_2,\cdots,e_n$ of $\in T^{1,0}_oM$, there is a local holomorphic coordinate $(z^1,z^2,\cdots,z^n)$ at $o$ such that $z(o)=0$, $e_i=\frac{\p}{\p z^i}\big|_o$,
$$\tilde \Gamma_{ij}^k(o)=\Gamma_{ij}^k(o)+\frac{1}{2}\tau_{ij}^k(o)=\frac12\left({\Gamma_{ij}^k+\Gamma_{ji}^k}\right)(o)=0$$
and
$$\p_i\tilde\Gamma_{jk}^l(o)X^iX^jX^k=\p_i\Gamma_{jk}^l(o)X^iX^jX^k=0$$
for any $X\in \C^n$.  We call this coordinate a normal local holomorphic coordinate at $o$ for the basis $(e_1,e_2,\cdots,e_n)$ of $T^{1,0}_oM$ on the Hermitian manifold.
\end{lem}
\begin{proof}
Let $\tilde D$ be the connection defined by $\tilde \Gamma_{ij}^k$. Then $\tilde D$ is a torsion free connection compatible with the complex structure. So, by standard argument, we know the existence of local holomorphic coordinate $w$ at $o$ such that $w(o)=0$, $e_i=\frac{\p}{\p w^i}\big|_o$ and $\tilde\Gamma_{w^iw^j}^{w^k}(o)=0$. Let $z$ be another local holomorphic coordinate at $o$ to be determined such that $z(o)=0$, $\frac{\p w^i}{\p z^j}(o)=\delta_j^i$ and $\frac{\p^2 w^k}{\p z^i\p z^j}(o)=0$. Then, it is clear that $e_i=\frac{\p}{\p z^i}\big|_o$ and by the transformation of Christofel symbols
\begin{equation}
\tilde\Gamma_{z^iz^j}^{z^k}=\tilde\Gamma_{w^\lambda w^\mu}^{w^\nu}\frac{\p w^\lambda}{\p z^i}\frac{\p w^\mu}{\p z^j}\frac{\p z^k}{\p w^\nu}+\frac{\p^2w^\nu}{\p z^i\p z^j}\frac{\p z^k}{\p w^\nu}
\end{equation}
we know that $\tilde \Gamma_{z^iz^j}^{z^k}(o)=0$. Moreover, note that
\begin{equation}
\begin{split}
\p_{z^i}\tilde\Gamma_{z^jz^k}^{z^l}(o)=&\p_{z^i}\tilde\Gamma_{w^\lambda w^\mu}^{w^\nu}\frac{\p w^\lambda}{\p z^j}\frac{\p w^\mu}{\p z^k}\frac{\p z^l}{\p w^\nu}+\frac{\p^3w^\nu}{\p z^i\p z^j\p z^k}\frac{\p z^l}{\p w^\nu}\\
=&\p_{w^i}\tilde\Gamma_{w^j w^k}^{w^l}(o)+\frac{\p^3w^l}{\p z^i\p z^j\p z^k}(0)
\end{split}
\end{equation}
So, we require that
\begin{equation}
\frac{\p^3w^l}{\p z^i\p z^j\p z^k}(0)X^iX^jX^k=-\p_{w^i}\tilde\Gamma_{w^j w^k}^{w^l}(o)X^iX^jX^k
\end{equation}
for any $X\in \C^n$. This requirement can  be satisfied by chosen $z$ such that
\begin{equation}
\frac{\p^3w^l}{\p z^i\p z^j\p z^k}(0)=-\frac13\left(\p_{w^i}\tilde\Gamma_{w^j w^k}^{w^l}+\p_{w^k}\tilde\Gamma_{w^i w^j}^{w^l}+\p_{w^j}\tilde\Gamma_{w^k w^i}^{w^l}\right)(o)
\end{equation}
by noting that $\tilde\Gamma_{w^iw^j}^{w_k}=\tilde\Gamma_{w^jw^i}^{w_k}$. This completes the proof of lemma.
\end{proof}
\begin{rem}
The requirement that $\p_i\tilde\Gamma_{jk}^l(o)X^iX^jX^k=0$ for any $X\in \C^n$ is the same as requiring that
\begin{equation}
\p_i\tilde\Gamma_{jk}^l+\p_j\tilde\Gamma_{ki}^l+\p_k\tilde\Gamma_{ij}^l(o)=0.
\end{equation}
It is also equivalent to
\begin{equation}
\p_i\Gamma_{jk}^l+\p_i\Gamma_{kj}^l+\p_j\Gamma_{ki}^l+\p_j\Gamma_{ik}^l+\p_k\Gamma_{ij}^l+\p_k\Gamma_{ji}^l(o)=0.
\end{equation}

\end{rem}
Next, we need the following lemma.
\begin{lem}\label{lem-expansion-z}
Let $(M^{2n},J,g)$ be a Hermitian manifold, $o\in M$, $e_1,e_2,\cdots,e_n$ be a unitary basis of $T_o^{1,0}M$,  and $z=(z^1,z^2,\cdots,z^n)$ be a normal local holomorphic coordinate for $(e_1,e_2,\cdots,e_n)$. Let $\gamma$ be a normal geodesic with $\gamma(0)=o$ and $\gamma'(0)=X=X^ie_i+X^{\bi}\ol{e_i}\in T_oM$. Suppose $z(\gamma(t))=(z^1(t),z^2(t),\cdots, z^n(t))$. Then,
\begin{equation*}
\begin{split}
&z^i(t)\\
=&X^i t+\frac12\tau_{\bar i\bar j}^\bk(o)X^\bj X^kt^2+\frac16\bigg(\left(R_{j\bi k\bl}+\frac12\tau_{\lambda j}^i\tau_{\bar\lambda\bar l}^\bk -\tau_{\bar \lambda \bi}^\bk\tau_{\lambda j}^l\right)(o) X^j X^k X^{\bar l}\\
&-\left(\p_\mu(\tau_{\bar j\bi}^\bk g_{l\bk})(o)X^\mu X^l
+\p_{\bar \mu}(\tau_{\bar j\bi}^\bk g_{l\bk})(o)X^{\bar\mu}X^l+\frac32\tau_{\bar j\bi}^\bk\tau_{\bar k\bar\lambda}^{\bar \mu}(o)X^{\bar\lambda}X^\mu\right) X^{\bar j}\bigg)t^3+O(t^4).
\end{split}
\end{equation*}
\end{lem}
\begin{proof}
Substituting the initial data $z^i(0)=0$ and $\frac{dz^i}{dt}(0)=X^i$ into \eqref{eq-geodesic}, we have
\begin{equation}\label{eq-2nd}
\begin{split}
\frac{d^2z^i}{dt^2}(0)=&-\Gamma_{jk}^i(o)X^jX^k-g^{\bar \lambda i}\tau_{\bar j\bar\lambda}^\bk g_{l\bk}(o)X^{\bar j}X^l\\
=&\frac12\tau_{jk}^i(o)X^jX^k-\tau_{\bar j\bar i}^\bk(o)X^{\bar j}X^k\\
=&\tau_{\bar i\bar j}^\bk(o)X^\bj X^k\\
\end{split}
\end{equation}
by that $\tilde \Gamma_{ij}^k(o)=0$ and $g_{i\bj}(o)=\delta_{ij}$. Moreover,
\begin{equation}\label{eq-3nd-1}
\begin{split}
&\frac{d^3z^i}{dt^3}(0)\\
=&-\p_l\Gamma_{jk}^i(o)\frac{dz^l}{dt}\frac{dz^j}{dt}\frac{dz^k}{dt}(0)-\p_{\bar l}\Gamma_{jk}^i(o)\frac{dz^{\bar l}}{dt}\frac{dz^j}{dt}\frac{dz^k}{dt}(0)+\frac12\tau_{jk}^i(o)\frac{d}{dt}\left(\frac{dz^j}{dt}\frac{dz^k}{dt}\right)(0)\\
&-\p_\mu(g^{\bar \lambda i}\tau_{\bar j\bar\lambda}^\bk g_{l\bk})(o)\frac{dz^\mu}{dt}\frac{dz^{\bar j}}{dt}\frac{dz^l}{dt}(0)-\p_{\bar \mu}(g^{\bar \lambda i}\tau_{\bar j\bar\lambda}^\bk g_{l\bk})(o)\frac{dz^{\bar \mu}}{dt}\frac{dz^{\bar j}}{dt}\frac{dz^l}{dt}(0)\\
&-g^{\bar \lambda i}\tau_{\bar j\bar\lambda}^\bk g_{l\bk}(o)\frac{d^2z^{\bar j}}{dt^2}\frac{dz^l}{dt}(0)-g^{\bar \lambda i}\tau_{\bar j\bar\lambda}^\bk g_{l\bk}(o)\frac{dz^{\bar j}}{dt}\frac{d^2z^l}{dt^2}(0)\\
=&R_{j\bi k\bl}(o)X^jX^kX^{\bar l}-\p_\mu(g^{\bar \lambda i})\tau_{\bar j\bar\lambda}^\bk (o)X^\mu X^{\bar j}X^k-\p_\mu(\tau_{\bar j\bi}^\bk g_{l\bk})(o)X^\mu X^{\bar j}X^l\\
&-\p_{\bar \mu}(g^{\bar \lambda i})\tau_{\bar j\bar\lambda}^\bk(o)X^{\bar\mu}X^{\bar j}X^k-\p_{\bar \mu}(\tau_{\bar j\bi}^\bk g_{l\bk})(o)X^{\bar\mu}X^{\bar j}X^l-\tau_{\bar \lambda \bi}^\bk\tau_{\lambda j}^l(o) X^j X^k X^{\bar l}\\
&-\tau_{\bar j\bi}^\bk\tau_{\bar k\bar\lambda}^{\bar \mu}(o)X^{\bar\lambda}X^\mu X^{\bar j}
\end{split}
\end{equation}
by noting that $\p_k\Gamma_{ij}^l(o)X^iX^jX^k=0$, $R_{j\bi k\bl}(o)=-\p_{\bar l}\Gamma_{jk}^i(o)$ and using \eqref{eq-2nd}. Furthermore, note that
\begin{equation}
\p_\mu g^{\bar \lambda i}(o)=-\p_\mu g_{\lambda\bi}(o)=-\Gamma_{\lambda\mu}^i(o)=\frac12\tau_{\lambda\mu}^i(o)
\end{equation}
and
\begin{equation}
\p_{\bar \mu} g^{\bar\lambda i}(o)=-\p_{\bar \mu} g_{\lambda \bi}(o)=-\Gamma_{\bi\bar\mu}^{\bar\lambda}(o)=\frac12\tau_{\bi\bar\mu}^{\bar\lambda}(o).
\end{equation}
Substituting these into \eqref{eq-3nd-1}, we get
\begin{equation}\label{eq-3nd-2}
\begin{split}
&\frac{d^3z^i}{dt^3}(0)\\
=&\left(R_{j\bi k\bl}+\frac12\tau_{\lambda j}^i\tau_{\bar\lambda\bar l}^\bk -\tau_{\bar \lambda \bi}^\bk\tau_{\lambda j}^l\right)(o) X^j X^k X^{\bar l}\\
&-\left(\p_\mu(\tau_{\bar j\bi}^\bk g_{l\bk})(o)X^\mu X^l
+\p_{\bar \mu}(\tau_{\bar j\bi}^\bk g_{l\bk})(o)X^{\bar\mu}X^l+\frac32\tau_{\bar j\bi}^\bk\tau_{\bar k\bar\lambda}^{\bar \mu}(o)X^{\bar\lambda}X^\mu\right) X^{\bar j}.
\end{split}
\end{equation}
Then, by Taylor expansion, we get the conclusion.
\end{proof}
We are now ready to prove Theorem \ref{thm-con-1-holo}.
\begin{proof}[Proof of Theorem \ref{thm-con-1-holo}]
Suppose the conclusion is not true. Then, there is a point $o\in M$ and a unitary basis $e_1,e_2,\cdots,e_n$ at $o$ such that
\begin{equation}
R_{1\bar11\bar1}+\sum_{i=2}^n|\tau_{i1}^1|(o)<0.
\end{equation}
Let $(z^1,\cdots,z^n)$ be a normal local holomorphic coordinate for $(e_1,e_2,\cdots,e_n)$ at $o$.

In the following, let the notations be the same as in Lemma \ref{lem-expansion-z}. Let $$f(z)=z^1\left(1-\frac12\tau_{1i}^{1}(o)z^i\right)$$ and $$p(z)=z^1\left(1-\frac12\tau_{\bar 1\bi}^{\bar 1}(o)z^\bi\right).$$ Then $f$ is a holomorphic function near $o$ and $|f|=|p|$. Because $\log M(|f|,r)$ is a convex function of $\log r$, by (4) of Theorem \ref{thm-mono-2}, we know that $\frac{M_o(|f|,r)}{r} $ is increasing. On the other hand, by Lemma \ref{lem-expansion-z},
\begin{equation*}
\begin{split}
&p(z(t))\\
=&X^1t+\frac12\left(\sum_{j,k=2}^n\tau_{\bar 1\bj}^\bk(o) X^\bj X^k\right)t^2+\\
&\frac16\bigg(\left(R_{j\bar 1 k\bl}+\frac12\tau_{\lambda j}^1\tau_{\bar\lambda\bar l}^\bk -\tau_{\bar \lambda \bar 1}^\bk\tau_{\lambda j}^l\right)(o) X^j X^k X^{\bar l}+\frac32\tau_{\bi 1}^{\bar 1}\tau_{ij}^k(o)X^jX^\bk X^1 \\
&-\left(\p_\mu(\tau_{\bar j\bar 1}^\bk g_{l\bk})(o)X^\mu X^l
+\p_{\bar \mu}(\tau_{\bar j\bar 1}^\bk g_{l\bk})(o)X^{\bar\mu}X^l+\frac32\tau_{\bar j\bar 1}^\bk\tau_{\bar k\bar\lambda}^{\bar \mu}(o)X^{\bar\lambda}X^\mu+\frac32\tau_{\bar 1\bar i}^{\bar 1}\tau_{\bar 1\bj}^\bk X^k X^\bi\right) X^{\bar j}\bigg)t^3\\
&+O(t^4).
\end{split}
\end{equation*}
Thus,
\begin{equation*}
\begin{split}
&|f(z(t))|^2\\
=&|p(z(t))|^2\\
=&|X^1|^2t^2+\Re\left\{\left(\sum_{j,k=2}^n\tau_{\bar 1\bj}^\bk(o) X^\bj X^k\right)X^{\bar 1}\right\}t^3\\
&+\frac13\Re\bigg\{\left(R_{j\bar 1 k\bl}+\frac12\tau_{\lambda j}^1\tau_{\bar\lambda\bar l}^\bk -\tau_{\bar \lambda \bar 1}^\bk\tau_{\lambda j}^l\right)(o) X^j X^k X^{\bar l}X^{\bar 1}+\frac32\tau_{\bi 1}^{\bar 1}\tau_{ij}^k(o)X^jX^\bk |X^1|^2 \\
&-\left(\p_\mu(\tau_{\bar j\bar 1}^\bk g_{l\bk})(o)X^\mu X^l
+\p_{\bar \mu}(\tau_{\bar j\bar 1}^\bk g_{l\bk})(o)X^{\bar\mu}X^l+\frac32\tau_{\bar j\bar 1}^\bk\tau_{\bar k\bar\lambda}^{\bar \mu}(o)X^{\bar\lambda}X^\mu+\frac32\tau_{\bar 1\bar i}^{\bar 1}\tau_{\bar 1\bj}^\bk X^k X^\bi\right) X^{\bar j}X^{\bar 1}\bigg\}t^4\\
&+\frac14\left|\sum_{j,k=2}^n\tau_{\bar 1\bj}^\bk(o) X^\bj X^k\right|^2t^4+O(t^5)\\
\leq&|X^1|^2t^2+\frac13(R_{1\bar 11\bar 1}+\sum_{i=2}^n|\tau_{i 1}^{ 1}|^2)(o)|X^1|^4t^4+C\left(\left(\frac12-|X^1|^2\right)t^3+t^4\sqrt{\frac12-|X^1|^2}+t^5\right)
\end{split}
\end{equation*}
for some positive constant $C$ and when $t$ is small enough, by noting  that $$\sum_{i=1}^n|X^i|^2=\frac12$$
since $\|X\|=1$ and $g_{i\bar j}(o)=\delta_{ij}$. Note that
\begin{equation}
\sqrt{\frac12-|X^1|^2}\leq \frac{\e}{4}+\frac1\e\left(\frac12-|X^1|^2\right)
\end{equation}
where $\e>0$ is chosen to satisfy
\begin{equation}
\frac{C\e}{4}=-\frac{1}{24}(R_{1\bar 11\bar 1}+\sum_{i=2}^n|\tau_{i1}^{1}|^2)(o).
\end{equation}
So, when $t$ is small enough,
\begin{equation}
\begin{split}
|f(z(t))|^2\leq& |X^1|^2t^2+\frac13(R_{1\bar 11\bar 1}+\sum_{i=2}^n|\tau_{i1}^{1}|^2)(o)|X^1|^4t^4\\
&+C\left(\left(\frac12-|X^1|^2\right)t^3+\left(\frac\e4+\frac1\e\left(\frac12-|X^1|^2\right)\right)t^4+t^5\right)
\end{split}
\end{equation}
Note that, when $t$ is small enough, the function
\begin{equation}
F_t(x)=x+\frac13(R_{1\bar 11\bar 1}+\sum_{i=2}^n|\tau_{i1}^{1}|^2)(o)x^2t^2+C\left(\left(\frac12-x\right)t+\left(\frac{\e}{4}+\frac1\e\left(\frac12-x\right)\right)t^2+t^3\right)
\end{equation}
is increasing for $x\in [0,\frac12]$. Thus $F_t(x)\leq F_t(\frac12)$. So,
\begin{equation}
|f(z(t))|^2\leq \frac{t^2}{2}+\frac1{24}(R_{1\bar 11\bar 1}+\sum_{i=2}^n|\tau_{i 1}^{1}|^2)(o)t^4+Ct^5< \frac{t^2}{2}
\end{equation}
when $t>0$ is small enough. This implies that
\begin{equation}
\frac{M_o(|f|,r)}{r}<\frac1{\sqrt 2}
\end{equation}
when $r>0$ is small enough. On the other hand, it is clear that
\begin{equation}
\lim_{r\to 0^+}\frac{M_o(|f|,r)}{r}=\frac1{\sqrt 2}.
\end{equation}
This contradicts the fact that $\frac{M_o(|f|,r)}{r}$ is increasing.

The equivalence of the two curvature conditions can be seen in the proof of Lemma \ref{lem-curv-holo-sec}.
\end{proof}

\end{document}